\documentclass{amsart}

\usepackage{amsmath,amsthm,amssymb,amscd}
\usepackage{tikz}
\usepackage[all,cmtip,color]{xy}
\usepackage[german,english]{babel}
\usepackage{amsmath}
\usepackage{hyperref}
\usepackage{mathrsfs} 
\usepackage{todonotes}
\usepackage[all,cmtip]{xy}
\usepackage{geometry}
\usepackage{enumerate}
\usepackage{mathtools}
\usepackage{wasysym}
\usepackage{colonequals}
\usepackage[export]{adjustbox}
\usepackage{graphicx}
\usepackage{wrapfig}
\usepackage{units}

\DeclareMathOperator{\Ho}{\mathsf{Ho}}

\DeclareMathOperator{\Cb}{\mathbb{C}}

\newcommand{\Cc}{\mathcal{C}}
\newcommand{\Gb}{\mathbb{G}}

\newcommand{\Dc}{\mathcal{D}}

\DeclareMathOperator{\Spaces}{\mathsf{Space}}

\DeclareMathOperator{\eff}{eff}

\DeclareMathOperator{\colim}{\operatorname{colim}}
\renewcommand{\lim}{\operatorname{lim}}

\newcommand{\C}{\mathsf{C}}

\newcommand{\Tf}{\mathsf{T}}

\DeclareMathOperator{\id}{id}

\DeclareMathOperator{\Sb}{\mathbb{S}}

\DeclareMathOperator{\Qb}{\mathbb{Q}}
\DeclareMathOperator{\Zb}{\mathbb{Z}}

\DeclareMathOperator{\MHS}{\mathsf{MHS}}

\DeclareMathOperator{\Ab}{\mathbb{A}}

\DeclareMathOperator{\Segal}{\mathsf{Segal}}

\DeclareMathOperator{\bip}{\mathsf{Split}}
\DeclareMathOperator{\loc}{\mathsf{loc}}

\DeclareMathOperator{\Aut}{\mathsf{Aut}}

\DeclareMathOperator{\G}{\mathbb{G}}

\DeclareMathOperator{\GL}{GL}
\DeclareMathOperator{\BGL}{BGL}

\DeclareMathOperator{\Var}{\mathsf{Var}}

\DeclareMathOperator{\Oo}{\mathcal{O}}

\DeclareMathOperator{\Lb}{\mathbb{L}}

\begin{document}

\theoremstyle{definition}
\newtheorem{definition}{Definition}[section]
\newtheorem{rmk}[definition]{Remark}

\newtheorem{example}[definition]{Example} 

\newtheorem{ass}[definition]{Assumption}
\newtheorem{warning}[definition]{Dangerous Bend}
\newtheorem{porism}[definition]{Porism}
\newtheorem{hope}[definition]{Hope}
\newtheorem{situation}[definition]{Situation}
\newtheorem{construction}[definition]{Construction}

\theoremstyle{plain}

\newtheorem{theorem}[definition]{Theorem}
\newtheorem{proposition}[definition]{Proposition}
\newtheorem{corollary}[definition]{Corollary}
\newtheorem{conj}[definition]{Conjecture}
\newtheorem{lemma}[definition]{Lemma}
\newtheorem{claim}[definition]{Claim}
\newtheorem{cl}[definition]{Claim}
\newtheorem*{acknowledgement}{Acknowledgement}

\title[Classes in Zakharevich $K$-groups]{Classes in Zakharevich $K$-groups constructed\linebreak from Quillen $K$-theory}
\author{Oliver Braunling \and Michael Groechenig}
\address{Institute for Mathematics, Albert Ludwig University Freiburg}
\email{oliver.braeunling@math.uni-freiburg.de}
\address{Department of Mathematics, University of Toronto}
\email{michael.groechenig@utoronto.ca}
\begin{abstract}
We show that the $K$-groups $K_{n}(\Oo)$ for $\Oo$ the integers or an order in a CM field and $n\geq 1$ appear as direct summands of the homotopy groups of various localisations of Zakharevich's $K$-theory space. After rationalisation and going to the $1$-connective cover, this even becomes a retract of spaces. As an application, we provide the first construction of classes of infinite order in the higher Zakharevich $K$-groups.
\end{abstract}
\thanks{
O. B. was supported by DFG GK1821 \textquotedblleft Cohomological Methods
in Geometry\textquotedblright\ . M. G. was supported by an NSERC discovery grant.
}

\maketitle

\section{Introduction}

The Grothendieck ring of complex varieties $K_0(\Var)$ is defined to be as the quotient of the free abelian group generated by isomorphism classes of complex varieties $[X]$, modulo the relation
$$[X] = [U] + [Z],$$
for every closed immersion $Z \hookrightarrow X$ with open complement $U$. In \cite{zak}, Zakharevich introduced a space $K(\Var)$ such that $\pi_0 K(\Var)$ is $K_0(\Var)$. The higher Zakharevich $K$-groups, defined as the homotopy groups ${K_i(\Var):={\pi_i}K(\Var)}$, remain mysterious. Ony few classes have been constructed explicitly, and all known ones are of finite order. In this note, we settle the existence of infinite order elements, and furthermore exhibit a connection between the Quillen $K$-theory of certain rings and Zakharevich $K$-theory.

To state our main result, we must introduce some notation. Let $(X,x_0)$ be a pointed space. We denote by $X^{\circ}_{x_0}$ the path-connected component of $x_0$ in $X$. The localisation $K(\Var)_{\loc}=K(\Var)[\G_m^{-1}]$ is defined as the homotopy colimit of $[K(\Var) \xrightarrow{\times \G_m} K(\Var) \xrightarrow{\times \G_m} \cdots]$. Pathconnected component of this space are in bijection with $K_0(\Var)_{\loc}=K_0(\Var)[\G_m^{-1}]$, the localisation of the Grothendieck ring in the element $[\G_m]$.

\begin{theorem} For $\Var$ referring to varieties over the complex numbers, the following holds:

\begin{enumerate}
\item For every $n\geq1$ the group $\pi_{n}\left(  K(\mathsf{Var})
[\mathbb{G}_{m}^{-1}]\right)$ has the Quillen $K$-group $K_{n}(\mathbb{Z})$ as a
direct summand.

\item The rational homotopy type $L_{\mathbb{Q}}K(\mathbb{Z})$ is a retract of
$L_{\mathbb{Q}}\left(  K(\mathsf{Var})[\mathbb{G}_{m}^{-1}]\right)  $ after
going to the $1$-connective cover.
\end{enumerate}
\end{theorem}

This will be Theorem \ref{thm1} in the main body of the article. Its proof relies on a exponential-like device, a map $i\colon K(\Zb) \to K(\Var)_{\loc}$, which does not induce a group homomorphism on the level of $\pi_0$. Instead, it intertwines $\oplus$ on $K(\Zb)$ with $\times$ on $K(\Var)_{\loc}$. Despite its unorthodox nature, this morphism is useful for constructing new classes in higher Zakharevich $K$-groups. Indeed, as a continuous map, $i$ gives rise to homomorphisms $K_i(\Zb) \to \pi_i(K(\Var)_{\loc})$ for $i>0$.

The construction of a retraction to $i$ (or left-inverse) hinges on the mixed Hodge realisation which the authors developed in joint work with Nanavaty \cite{paperletto}.

Since elements of $K_i(\Var)_{\loc}$ can be represented as fractions, it is possible to draw the following conclusion from the statement above and Borel's study of $K(\Zb)$ (see \cite{borel}):

\begin{corollary}\label{announce_cor_in_intro}
\label{cor_intro1}For every positive integer $n$, there exists an element of
infinite order in $K_{4n+1}(\mathsf{Var})$ for varieties over the complex numbers.
\end{corollary}

This will be Corollary \ref{cor}.

Such an element is automatically \emph{indecomposable} in the sense of \cite[Section 7]{cwz}. This provides an affirmative answer to Question 7.1 in \cite{cwz}.

We arrive at the proof of Corollary \ref{announce_cor_in_intro} on page \pageref{cor}. The remaining part of the paper will be devoted to a
technically more sophisticated variant of the same homotopical constructions.
It replaces the algebraic torus $\mathbb{G}_{m}$ by an abelian variety with
complex multiplication.

We recall that a \emph{CM\ field} $F$ is a number field which is a totally
complex quadratic extension of a totally real number field, e.g.
$F=\mathbb{Q}(\sqrt{-m})$ for any $m\geq1$.

\begin{theorem}
Let $F$ be a CM field with ring of integers $\mathcal{O}$. Let $A$ be a simple
abelian variety over a perfect field $k$ (possibly of positive characteristic) with
complex multiplication by $\mathcal{O}$. Let $\mathsf{Var}$ denote varieties
over $k$.

\begin{enumerate}
\item For every $n\geq1$ the group $\pi_{n}\left(  K(\mathsf{Var})
[A^{-1}]\right)$ has the Quillen $K$-group $K_{n}(\mathcal{O})$ as a direct summand.

\item The rational homotopy type $L_{\mathbb{Q}}K(\mathcal{O})$ is a retract
of $L_{\mathbb{Q}}\left(  K(\mathsf{Var})[A^{-1}]\right)  $ after going to the
$1$-connective cover.
\end{enumerate}
\end{theorem}

This will be Theorem \ref{thm2} in the main body of the article. In order to
be able to handle fields of positive characteristic, we replace the mixed
Hodge realisation by a realisation to Deligne $1$-motives based on the derived
Albanese functor of Barbieri-Viale and Kahn \cite{bv}. In particular, this
allows us to strengthen Corollary \ref{cor_intro1} to the following.

\begin{corollary}
Suppose $k=\mathbb{C}$. We have $\dim\left(  K_{n}(\mathsf{Var})\otimes
\mathbb{Q}\right)  =\infty$ for all odd $n\geq3$.
\end{corollary}

The restriction to the field of complex numbers can be removed by working with abelian varieties with complex multiplication. This requires replacing mixed Hodge structures by $1$-motives. 

\begin{theorem}\label{thm3_intro}
Let $A$ be a simple
abelian variety over a perfect field $k$ with complex multiplication by some order $\mathcal{O} \subseteq F$, where $F$ is a division algebra (e.g. an order in a quaternion algebra). Then, the rational homotopy type $L_{\mathbb{Q}}K(\mathcal{O})$ is a retract
of $L_{\mathbb{Q}}\left(  K(\mathsf{Var})[A^{-1}]\right)  $ after going to the
$1$-connective cover.
\end{theorem}

This will be Theorem \ref{thm3}. 

\begin{acknowledgement}
We thank Anubhav Nanavaty for the joint work which led to \cite{paperletto}, which is a crucial input for the present article. It's a pleasure to thank Inna Zakharevich for her inspiring talk about constructing interesting classes in higher Zakharevich $K$-theory \cite{zaktalk}.
\end{acknowledgement}

\section{Preliminaries}\label{sec:prelim}

\subsection{Conventions}
We use the theory of $\infty$-categories as described by Lurie \cite{ha}. We use this as our principal model for homotopy theory. We emphasise that most of our constructions do not rely on the $\infty$-categorical structures and work solely on the level of homotopy category.

We freely identify categories with their nerves so that groupoids can always be understood as homotopy $1$-types.

\subsection{\texorpdfstring{$K$}{K}-theory of symmetric monoidal categories}

Let $\Cc$ be a category. We denote by $\Cc^{\times}$ the maximal subgroupoid (also known as \emph{core}), that is, the groupoid obtained by discarding all non-invertible morphisms.

If $\Cc$ is endowed with a symmetric monoidal structure $\otimes$, then $(\Cc^{\times},\otimes)$ is a symmetric monoidal groupoid. By virtue of the nerve construction, $(\Cc^{\times},\otimes)$ gives rise to a Segal $\Gamma$-space\footnote{We follow the convention of Segal's original definition. In alternative nomenclature these are called \textit{special} $\Gamma$-spaces. At any rate the underlying space of our $\Gamma$-spaces has a $H$-space structure.} (or equivalently, an $E_{\infty}$-object in $\Spaces$)
$\Cc^{\times}_{\Gamma}.$

Recall that there is an idempotent endo-functor 
$$\Omega B\colon \Ho(\Segal) \to  \Ho(\Segal)$$
sending a Segal $\Gamma$-space to its group-completion. Furthermore, there is a natural transformation $$c\colon \id \Rightarrow \Omega B$$ such that 
$$c_X\colon X \to \Omega B X$$
is an equivalence if and only if $X$ is group-like.

\begin{definition}
The $K$-theory of $(\Cc,\otimes)$ is defined to be $K^{\otimes}(\Cc) = \Omega B\Cc^{\times}_{\Gamma}.$ By construction, there is a morphism
$$c\colon \mathcal{C}^{\times} \to K^{\otimes}(\mathcal{C}).$$
\end{definition}

Let $R$ be a commutative ring with unit. We denote by $P_f(R)$ the category of finitely generated projective $R$-modules. The direct sum operation for $R$-modules defines a symmetric monoidal structure $\oplus$ on $P_f(R)$. 

\begin{definition}
$K^{\oplus}(R) = K^{\oplus}(P_f(R)^{\times}).$
\end{definition}

For $\Zb$ is a principal ideal domain we have $P_f(\Zb) = F_f(\Zb)$, the category of finitely generated free abelian groups. We may therefore restate the definition above as follows:
$$K(\Zb) = K^{\oplus}(F_f^{\times}(\Zb)).$$

The pathconnected space $K(\Zb)^{\circ}_0$ can be described as the $+$-construction of $\BGL(\Zb)$, where $GL(\Zb)$ denotes the group $\bigcup_n \GL_n(\Zb)$. The latter satisfies a (weak) universal property for maps into a group-like $H$-spaces (see \cite[Theorem 1.8 \& Remark 1.8.1]{K-book} or \cite[Theorem 2.5]{MR0382398}). We denote by $k\colon \BGL(\Zb) \to \BGL(\Zb)^+$ the canonical map given by the $+$-construction.
\begin{theorem}[Quillen]\label{thm:magic}
Let $H$ be an $H$-space and $f\colon \BGL(\Zb) \to H$ be a map. Then, there exists $g\colon \BGL(\Zb)^+ \to H$ such that 
\[
\xymatrix{
 \BGL(\Zb) \ar[r]^-k \ar[rd]_f &  \BGL(\Zb)^+ \ar@{-->}[d]^g \\
 & H
}
\]
commutes up to homotopy. Furthermore, the induced graded homomorphism $\pi_*(g)\colon \pi_*(\BGL(\Zb)^+) \to \pi_*(H)$ is independent of the choice of $g$.
\end{theorem}

\subsection{Waldhausen \texorpdfstring{$K$}{K}-theory}

A convenient way of defining $K$-theory spaces is given by Waldhausen's $S$-construction, which was introduced in \cite{wald}. We refer to \emph{loc. cit.} for detailed definitions and content ourselves with a brief review of the main ideas. 

A Waldhausen category $\Cc$ is a category with two specified subclasses of morphisms that are called \emph{cofibrations} and \emph{weak equivalences}. The axioms of \emph{loc. cit.} are required to hold. Waldhausen defines a simplicial category $S_{\bullet}\Cc$, by stipulating $S_{n}\Cc$ to consist of chains of cofibrations of the following shape when $n>0$:
$$X_0 \hookrightarrow \cdots \hookrightarrow X_{n-1}.$$
and $S_0\Cc$ to consist of a single object without non-trivial endomorphisms. Degeneracy maps are defined by repeating an object
$$(X_0 \hookrightarrow \cdots \hookrightarrow X_{n-1}) \mapsto (X_0 \hookrightarrow \cdots \hookrightarrow X_i \xrightarrow{\id} X_i \hookrightarrow \cdots \hookrightarrow X_{n-1}),$$
and face maps are given by omission/composition of subsequent cofibrations, and a quotient construction which is guaranteed to exist by Waldhausen's axioms:
$$(X_0 \hookrightarrow \cdots \hookrightarrow X_{n-1}) \mapsto (X_1/X_0 \hookrightarrow \cdots \hookrightarrow X_{n-1}/X_0).$$
\begin{definition}[Waldhausen]
$K(\Cc) = \Omega|(S_{\bullet}\Cc)^{\times w}|$,
where $(-)^{\times w}$ means that we restrict to the groupoid whose morphisms are the weak equivalences. In particular, if the weak equivalences of the Waldhausen category structure are merely the isomorphisms, this agrees with $\Cc^{\times}$.
\end{definition}

\begin{example}
\label{example_ExactCatsAreWaldhausen}Suppose $\mathcal{C}$ is an exact
category and fix a zero object $0$. Then, taking the admissible monomorphisms
as the cofibrations and isomorphisms as weak equivalences equips $\mathcal{C}$
with a Waldhausen category structure. Hence, whenever we write $K(\mathcal{C}%
)$ for an exact category $\mathcal{C}$, this is the Waldhausen structure we
mean. Exact functors between exact categories become exact functors of
Waldhausen categories.
\end{example}

\begin{example}
\label{example_DGCatsAreWaldhausen}If $\mathcal{C}$ is a DG\ category, the
category of DG\ modules $\operatorname*{mod}(\mathcal{C}^{op})$ carries a standard
model structure (\cite[Subsection 5.2(a)]{toen}). If
$\mathcal{C}$ is a\ DG\ category, the full subcategory of cofibrant and
compact DG\ modules $\operatorname*{perf}_{co}(\mathcal{C}^{op})\subset
\operatorname*{mod}(\mathcal{C}^{op})$ carries the structure of a Waldhausen
category whose cofibrations and weak equivalences are those of the
aforementioned model structure. We refer to \cite[\S 2.2]{paperletto} for details.
\end{example}

\begin{example}
If $\mathcal{C}$ is an exact category and $\mathsf{C}^{b}(\mathcal{C})$ its
DG\ category of bounded complexes, we have $K(\mathcal{C})\cong K(\mathsf{C}%
^{b}(\mathcal{C}))$ by the Gillet--Waldhausen theorem. In this sense, the
previous examples are compatible.
\end{example}

\begin{theorem}[Waldhausen]
Let $R$ be a commutative ring with unit. Then, $P_f(R)$ can be endowed with the structure of a Waldhausen category such that $K^{\oplus}(R) = K(P_f(R)^{\times}).$
\end{theorem}
We refer the reader to \cite[Section 1.6.5]{wald} for a proof.

\begin{construction}\label{const:map}
Let $\Cc$ be a Waldhausen category. There is a natural map $\Cc^{\times} \to K(\Cc)$.
\end{construction}
\begin{proof}[Details]
By virtue of definition, $\Sigma K(\Cc)=BK(\Cc)$ completes $S_{\bullet}\Cc^{\times}$ to an augmented simplicial object:
$$S_{\bullet}\Cc^{\times} \to \Sigma K(\Cc).$$
Truncating this augmented simplicial diagram, we obtain a commutative square
\[
\xymatrix{
S_1\Cc^{\times} \ar[r] \ar[d] & \{\bullet\} \ar[d] \\
\{\bullet\} \ar[r] & \Sigma K(\Cc)\text{.}
}
\]
The universal property of pullbacks in the $\infty$-category of spaces yields a map $S_\Cc^{\times} \to \Omega\Sigma K(\Cc) \simeq K(\Cc)$. Since $S_1\Cc \cong \Cc$, we obtain the map we wanted.
\end{proof}

\begin{lemma}\label{lemma:comm1}
Let $\Cc \to \Dc$ be an exact functor of Waldhausen categories. Then, there is a homotopy commutative diagram
\[
\xymatrix{
\Cc^{\times} \ar[r] \ar[d] & \Dc^{\times} \ar[d] \\
K(\Cc) \ar[r] & K(\Dc)
}
\]
\end{lemma}
\begin{proof}
Omitted.
\end{proof}

\subsection{Zakharevich \texorpdfstring{$K$}{K}-theory and Campbell's \texorpdfstring{$S$}{S}-construction}

The original viewpoint on the $K$-theory space of varieties $K(\Var)$ that was developed in \cite{zak} is based on the categorical notion of \emph{assemblers}. In this article we will use an alternative viewpoint which is due to Campbell (see \cite{cam}). This approach is based on so-called \emph{subtractive} categories (and more generally, on $SW$-categories).

\begin{definition}[Campbell]
A \emph{category with subtraction} $\C$ is a category endowed with a notion of \emph{cofibrations} $\mathsf{cof}$ and \emph{fibrations} $\mathsf{fib}$. We denote arrows in $\mathsf{cof}$ by $\hookrightarrow$ and those in $\mathsf{fib}$ by $\xrightarrow{\circ}$. The following conditions must hold:
\begin{itemize}
\item[(a)] $\C$ has an initial object, called $\emptyset$.
\item[(b)] An isomorphism is a cofibration and a fibration,
\item[(c)] Pullbacks of cofibrations (respectively fibrations) exist and are again cofibrations (respectively fibrations),
\item[(d)] There exists a family of so-called \emph{subtraction sequences} $\{A \hookrightarrow B \xleftarrow{\circ} C\}$ satisfying several conditions (see \cite[Definition 3.7(4)]{cam}).
\end{itemize}
\end{definition}

\begin{definition}
A \emph{subtractive category} is given by a category with subtraction $\C$ satisfying the conditions:
\begin{itemize}
\item[(a)] Pushouts of cofibrations along cofibrations exist. Every morphism in the resulting pushout square is a cofibration, and the square itself is also required to be a pullback square.
\item[(b)] For a cartesian square of cofibrations
\[
\xymatrix{
A \ar@{^(->}[r] \ar@{^(->}[d] & B \ar@{^(->}[d] \\
C \ar@{^(->}[r] & D
}
\]
we have that the induced map $B \sqcup_A C \to D$ is a cofibration.
\item[(c)] For every commutative diagram
\[
\xymatrix{
A' \ar@{^(->}[d]& \ar@{_(->}[l] B' \ar@{^(->}[d]\ar@{^(->}[r] & C' \ar@{^(->}[d] \\
A &  B \ar@{^(->}[r] \ar@{_(->}[l] & C \\
A'' \ar[u]^{\circ}& B'' \ar@{_(->}[l] \ar@{^(->}[r] \ar[u]^{\circ}& C''\ar[u]^{\circ}
}
\]
the pushouts along the rows yield a subtraction sequence
$$A'\sqcup_{B'} C' \hookrightarrow A\sqcup_B C \xleftarrow{\circ} A''\sqcup_{B''} C''\text{.}$$
\end{itemize}
\end{definition}

\begin{rmk}
We will say that a subtractive category is \emph{split}, if every subtraction sequence is isomorphic to a \emph{split subtraction sequence}
$$A \hookrightarrow A \sqcup C \xleftarrow{\circ} C.$$
\end{rmk}

In \cite[Definition 3.29]{cam}, Campbell defines an $S$-construction for subtractive categories. The definition resembles Waldhausen's, with the salient feature being that cofibre sequences are traded for subtraction sequences.
\begin{definition}[Campbell]
The $K$-theory space of a subtractive category $\C$ is defined as $K(\C)=\Omega |S_{\bullet}\C^{\times}|$. 
\end{definition}
We emphasise that $K(\C)$ is endowed with the structure of an infinite loop space (i.e. connective spectrum), according to \cite[Proposition 5.3]{cam}.

Furthermore, Campbell introduces in \cite[Definition 5.9]{cam} the notion of \emph{$W$-exact functors} $\C \to \mathcal{D}$ from a subtractive to a Waldhausen category. Such functors give rise to maps between $S$-constructions, and thus induce a map of $K$-theory spectra (see \cite[Proposition 5.3]{cam})
$$K(\C) \to K(\mathcal{D})\text{.}$$

In fact, we will use the broader concept of \emph{weakly $W$-exact functors} from \cite[Definition 2.17]{cwz}.
\begin{construction}\label{const:map2}
There is a map
$z\colon \C^{\times} \to K(\C)$ such that for every weakly $W$-exact functor $\C \to \mathcal{D}$ we have a homotopy commutative diagram
\[
\xymatrix{
\C^{\times} \ar[r] \ar[d] & {\mathcal{D}}^{\times} \ar[d] \\
K(\C) \ar[r] & K(\mathcal{D})\text{.}
}
\]
\end{construction}
\begin{proof}[Details]
This is analogous to Construction \ref{const:map} and Lemma \ref{lemma:comm1}. One uses that both flavours of the $S$-construction are compatible, which is true by design. The same idea was employed in \cite[Proposition 2.20]{cwz}.
\end{proof}

\section{Retracts from split tori}\label{section_torusconstruction}

\subsection{The Grothendieck spectrum of varieties}\label{sect:grothringspectrum}

Suppose we fix a field $k$. By a $k$-\emph{variety} we mean a reduced finite
type separated scheme over the field $k$. Morphisms between $k$-varieties will
always be assumed to be $k$-morphisms.

We write $\Var$ for the $SW$-category of $k$-varieties, as introduced
in \cite[\S 3]{cam}. The product of $k$-varieties%
\[
X\times Y:=X\times_{k}Y
\]
equips $\Var$ with a multiplicative structure in the sense of
\cite[\S 5.2]{cam}. We write $\mathbf{1}=\operatorname{Spec}(k)$ to denote the neutral element for this monoidal structure.

As is shown in \emph{loc. cit.}, this definition equips $K(\Var)$ with
the structure of an $E_{\infty}$-ring spectrum.


\subsection{Construction of the incoming map}

We recall that $F_{f}(R)$ denotes the category of finitely generated free $R$-modules. The
direct sum renders $(F_{f}(R),\oplus)$ a symmetric monoidal category.

Write $\mathsf{Tori}\subset\mathsf{Var}$ for the full category of split
algebraic $k$-tori (that is: the category of varieties which admit an
isomorphism to $\mathbb{G}_{m}^{n}$ for some $n\geq0$).

Then, there is a natural symmetric monoidal functor%
\[
(F_{f}(\mathbb{Z}),\oplus)^{\times}\longrightarrow(\mathsf{Tori}%
,\times)^{\times}%
\]
sending $\Zb$ to $\G_{m}$.

\begin{rmk}\label{rmk_WhatAreTheMapsForGm}
As an aside: This functor is faithful, but not full. The morphisms in its image are precisely those which preserve the neutral element of the group and hence are a group homomorphism. On the right every morphism can be decomposed into group homomorphisms and translations (for $\Gb_m$ a translation is just the multiplication with an element).
\end{rmk}

We will tacitly identify categories
with their nerves. Thus, the above defines a map of spaces. Recall that once
we restrict to the connected component of $\Zb^{n}$ on the left side,
we obtain the homotopy type of $\BGL_{n}(\Zb)$.

Using $\mathsf{Tori}\subset\mathsf{Var}$ and postcomposing with Construction \ref{const:map2}, we obtain maps
\[
\BGL_n(\Zb) \subset B\Aut_{\Var}(\G_m^n) \to K(\Var).
\]
Letting $\GL_n(\Zb) \subset \GL_{n+1}(\Zb)$ be the standard embedding
\[
A\mapsto%
\begin{pmatrix}  
A & 0\\
0 & 1
\end{pmatrix},
\]
we have a homotopy commutative diagram 
\begin{equation}\label{lcc1}
\xymatrix{
\BGL_n(\Zb) \ar[r] \ar[d] & \BGL_{n+1}(\Zb) \ar[d] \\
\big(K(\Var)_{\loc}\big)^{\circ}_{\G_m^n} \ar[r]^-{\times \G_m}_{\sim } & \big(K(\Var)_{\loc}\big)^{\circ}_{\G_m^{n+1}}
}
\end{equation}
of spaces. Using the canonical equivalence
\begin{equation}\label{lcc2}
\times \G_m^{-n}\colon \big(K(\Var)_{\loc}\big)^{\circ}_{\G_m^n} \simeq \big(K(\Var)_{\loc}\big)^{\circ}_{\mathbf{1}},
\end{equation}
we obtain a morphism of spaces
$$\tilde{i}\colon \BGL(\Zb) = \colim_n \BGL_n(\Zb) \to \big(K(\Var)_{\loc}\big)^{\circ}_{\mathbf{1}}.$$

%

\begin{lemma}\label{lemma:comm3}
There exists a morphism $i\colon K(\Zb)^{\circ}_0 \to (K(\Var)_{\loc},\times)$ of $H$-spaces such that the following diagram
\[
\xymatrix{
\BGL(\Zb) \ar[rd]^{\tilde{i}} \ar[d]_{c} &  \\
K(\Zb)^{\circ}_0 \ar@{-->}[r]^-i & \big(K(\Var)_{\loc}\big)^{\circ}_{\mathbf{1}}
}
\]
commutes.
\end{lemma}
\begin{proof}
This is a direct consequence of Theorem \ref{thm:magic} (the almost-universal property of the Plus-construction) applied to $\tilde{i}$.
\end{proof}

\subsection{Constructions from the theory of mixed Hodge structures}\label{subsect_MHSRealization}

Suppose $k=\mathbb{C}$. Let ${\MHS_{\Zb}}$ denote the category of
mixed Hodge structures with $\mathbb{Z}$-coefficients as introduced by
Deligne. This is an abelian category by \cite[Theorem 2.3.5]{delignehodge2}.

A mixed Hodge structure $M$ is called \emph{effective} if $F^{1}=0$. This
means that the Hodge decomposition of each graded weight piece
$\operatorname*{Gr}_{w}^{W}M$ decomposes as%
\[
M\otimes\mathbb{C\cong}\bigoplus_{p+q=w}M^{p,q}%
\]
with $p,q\leq0$.

\begin{rmk}
\label{remark_NoPositiveWeightsInEffectiveMHS}In particular, a mixed effective
Hodge structure has no graded pieces of strictly positive weight.
\end{rmk}

The full category ${\MHS^{\eff}_{\Zb}}\subset{\MHS_{\Zb}}$ of effective polarizable mixed Hodge structures is closed under
extensions and therefore the family of all exact sequences, restricted to this subcategory, equips ${\MHS^{\eff}_{\Zb}}$ with an exact structure (\cite[Lemma 10.20]{MR2606234}).
The inclusion as a subcategory then is an exact functor and thus yields a map%
\[
K({\MHS^{\eff}_{\Zb}})\longrightarrow K({\MHS_{\Zb}})
\]
in the sense of Example \ref{example_ExactCatsAreWaldhausen}.

In \cite[\S 5]{paperletto} a realisation functor%
\[
R:\mathsf{Var}\longrightarrow{\MHS^{\eff}_{\Zb}}%
\]
is constructed. It will be useful to recall a little bit how this is set up.
The cited paper first constructs a weakly $W$-exact functor to Voevodsky mixed
motives%
\begin{align}
\mathsf{Var}  & \longrightarrow\mathcal{DM}_{gm,\acute{e}t}^{\eff}%
(\mathbb{C};\mathbb{Z})\label{lz1}\\
V  & \longmapsto M^{c}(V)\nonumber
\end{align}
where $\mathcal{DM}_{gm,\acute{e}t}^{\eff}$ refers to the DG\ category of
\'{e}tale effective geometric mixed motives (over the complex numbers as the
base field and with $\mathbb{Z}$-coefficients) and $M^{c}(V):=z_{equi}(V,0)$
denotes the sheaf of equidimensional cycles. We will not recall what this is,
as we may use it as a blackbox for the purposes of this paper. It is a
concrete representative of the compactly supported motive of $X$%
.\footnote{Readers who are not familiar with this may imagine this as the
universal gadget which is able to recover the values of any compactly
supported cohomology theory of the variety.} As a second step, it composes
this construction with Vologodsky's construction of the mixed Hodge
realisation for motives (\cite[\S 2.8]{MR3004172}). The latter takes the shape of a DG
quasi-functor%
\begin{equation}
\mathcal{DM}_{gm,\acute{e}t}^{\eff}(\mathbb{C};\mathbb{Z})\longrightarrow
{\MHS^{\eff}_{\Zb}}\text{.}\label{lz2}%
\end{equation}
Even though the first functor from \eqref{lz1} goes from an
$SW$-category to a DG category and the latter, \eqref{lz2}, is between
DG\ categories, their composite induces a map%
\begin{equation}
K(\mathsf{Var})\longrightarrow K({\MHS^{\eff}_{\Zb}}%
)\text{.}\label{eq_HodgeRealiz}%
\end{equation}
We refer to \cite{paperletto} for the details of this construction.

\begin{example}
[{\cite[\S 5]{paperletto}}]On $K_{0}$ this functor sends a proper smooth
variety $V\in\mathsf{Var}$ to $\sum_{i}[H_{i}(V_{\mathbb{C}},\mathbb{Z})]$,
where $V_{\mathbb{C}}$ refers to the complex manifold attached to $V$ and each
homology group $H_{i}$ gets equipped with its canonical Hodge structure.
\end{example}


Let us quickly recall the categorical structures underlying K\"{u}nneth-type theorems. Recall that as part of the datum of a strong monoidal functor $F\colon
(\mathcal{A},\otimes,\mathbf{1})\rightarrow(\mathcal{A}^{\prime}%
,\otimes^{\prime},\mathbf{1}^{\prime})$ we are given a natural transformation
between the two functors%
\[
\mathcal{A}\times\mathcal{A}\longrightarrow\mathcal{A}^{\prime}%
\]
given by $F(-)\otimes^{\prime}F(-)$ and $F(-\otimes-)$ respectively. That is,
for any objects $W,X\in\mathcal{A}$ we are provided with an isomorphism%
\begin{equation}
\phi_{W,X}\colon F(W)\otimes^{\prime}F(X)\overset{\sim}{\longrightarrow
}F(W\otimes X)\label{l_nattransfstrongmonoidal}%
\end{equation}
satisfying the usual axioms.

There is a strong monoidal functor from the category of (at first smooth)
varieties $Sm$ with the usual product of varieties to Voevodsky mixed motives
with the DG\ tensor structure of \cite[\S 2.3]{MR2399083},%
\[
(Sm,\times_{k})\longrightarrow(\mathcal{DM}_{\acute{e}t}^{\operatorname*{eff}%
}(k;\mathbb{Z}),\otimes)\text{.}%
\]
The strong monoidal functor is implicit \textit{loc. cit.} (first
\cite[\S 2.2]{MR2399083} sets up a DG tensor structure on $\mathcal{P}%
_{\operatorname*{tr}}$ in the notation \textit{loc. cit.}, and establishes a
strong monoidal functor from $Sm$, and then \cite[\S 2.3]{MR2399083} forms
$\mathcal{DM}_{\acute{e}t}^{\operatorname*{eff}}(k;\mathbb{Z})$, along with
its DG\ tensor structure, from $\mathcal{P}_{\operatorname*{tr}}$ as a
quotient by a tensor ideal).

The extension of the functor $Sm\rightarrow\mathcal{DM}_{\acute{e}%
t}^{\operatorname*{eff}}(k;\mathbb{Z})$ to all varieties is uniquely
determined by enforcing $cdh$ descent \cite[\S 6.9.2, (i)]{MR2399083} (or
$ldh$ \cite{MR3673293} in the situation of positive characteristic and motives
with $\mathbb{Z}[\frac{1}{p}]$-coefficients\footnote{The remark about positive
characteristic is not relevant in this section as we work towards results over
the complex numbers, but we shall later return to this, so it is useful to
point this out already here.}). This induces the corresponding extension of
the monoidal functor%
\begin{equation}\label{l_DiscussCDHDescentExtension}
(\mathsf{Var},\times_{k})\longrightarrow(\mathcal{DM}_{\acute{e}%
t}^{\operatorname*{eff}}(k;\mathbb{Z}),\otimes)\text{.}%
\end{equation}
As a second step, there is a strong monoidal functor from mixed motives to
mixed Hodge structures with the standard tensor product of mixed Hodge
structures.%
\[
(\mathcal{DM}_{\acute{e}t}^{\operatorname*{eff}}(k;\mathbb{Z}),\otimes
)\longrightarrow(\mathsf{MHS}_{\mathbb{Z}}^{\operatorname*{eff}},\otimes)
\]
Composing both natural transformations and fixing the object $X$, this
restricts to a natural transformation (and indeed equivalence) of the two functors $\mathsf{Var}%
\rightarrow\mathsf{MHS}_{\mathbb{Z}}^{\operatorname*{eff}}$ given by%
\begin{equation}
W\mapsto R(W\times X)\qquad\text{and}\qquad W\mapsto R(W)\otimes
R(X)\text{.}\label{l_KuennethProperty}%
\end{equation}
We shall refer to this natural equivalence in the following proof.

The above only discussed the natural transformation in Equation
\ref{l_nattransfstrongmonoidal}. We leave an analogous discussion for the
natural transformation of the tensor units, which is also part of a monoidal
functor, to the reader.



\begin{lemma}\label{lemma_commute_realiz}
For every $X \in \Var$ we have a commutative diagram 
\[
\xymatrix{
K(\Var) \ar[r]^{X \times ? } \ar[d] & K(\Var) \ar[d] \\
K(\MHS_{\Zb}) \ar[r]^{R(X) \otimes ?} & K(\MHS_{\Zb})
}
\]
in the homotopy category of spaces.
\end{lemma}
\begin{proof}
(If it were in the literature that the mixed Hodge realisation is a map of $E_{\infty}$-ring spectra, this would be immediate. As we only need compatibility with the multiplication with a single variety, we get away with proving a much weaker statement.) If $F\colon\mathsf{C}\rightarrow\mathcal{D}$ is a weakly $W$-exact functor
from an $SW$-category $\mathsf{C}$ to a Waldhausen category $\mathcal{D}$, and
$G\colon\mathcal{D}\rightarrow\mathcal{D}^{\prime}$ an exact functor of
Waldhausen categories, then $G\circ F$ is again a weakly $W$-exact functor. If
$G\colon\mathcal{D}\rightarrow\mathcal{D}^{\prime}$ is a DG\ \textit{quasi-}%
functor between DG\ categories, we can canonically regard $\mathcal{D}$ and
$\mathcal{D}^{\prime}$ as Waldhausen categories through the device of Example
\ref{example_DGCatsAreWaldhausen}. Yet, it is not necessarily true that
$G\circ F$ is a genuine functor. However, we may by abuse of notation replace
${\MHS^{\eff}_{\Zb}}$ through a zig-zag of DG
equivalences\footnote{i.e. genuine DG\ functors which induce an equivalence of
DG categories.} by a DG\ category so that Vologodsky's mixed Hodge realisation%
\begin{equation}
r\colon\mathcal{DM}_{gm,\acute{e}t}^{\eff}(\mathbb{C};\mathbb{Z}%
)\longrightarrow{\MHS^{\eff}_{\Zb}}\label{lz3a}%
\end{equation}
is a genuine DG\ functor. Then, also the realisation followed by tensoring with $R(X)$,
i.e. the composition%
\begin{equation}
r_{X}\colon\mathcal{DM}_{gm,\acute{e}t}^{\eff}(\mathbb{C};\mathbb{Z}%
)\underset{r}{\longrightarrow}{\MHS^{\eff}_{\Zb}}\underset{\otimes
R(X)}{\longrightarrow}{\MHS^{\eff}_{\Zb}}\text{,}\label{lz3}%
\end{equation}
is a genuine DG\ functor (as the tensoring is an exact endofunctor of ${\MHS^{\eff}_{\Zb}}$). This replacement being done, $G\circ F$ is also
a weakly $W$-exact functor taking values in this particular model of
${\MHS^{\eff}_{\Zb}}$. \textit{(Step 1)}\ Consider the square in
the claim. We first consider the composition of maps going right and then
down, call this $RD$. This map steps from the multiplication%
\[
K(\mathsf{Var})\longrightarrow K(\mathsf{Var})
\]
coming from the $E_{\infty}$-ring structure discussed in
\S \ref{sect:grothringspectrum} and is then followed by the map of Equation
\ref{eq_HodgeRealiz}. We will provide an alternative description of this map.
Slightly modifying the construction of the mixed Hodge realisation functor, we
claim that%
\begin{align*}
R_{X}\colon\mathsf{Var}  & \longrightarrow\mathcal{DM}_{gm,\acute{e}t}%
^{\eff}(\mathbb{C};\mathbb{Z})\\
V  & \longmapsto M^{c}(V\times X)
\end{align*}
is a weakly $W$-exact functor. This can be shown by copying the proof in as much as
it differs from \eqref{lz1} just by the additional product with $X$
inside. Composing this with \eqref{lz3a} this defines a weakly
$W$-exact functor. Tracing what we have done, the induced map $K(\mathsf{Var}%
)\rightarrow K({\MHS^{\eff}_{\Zb}})$ agrees with $RD$.
\textit{(Step 2)} On the other hand, the unmodified mixed Hodge realisation%
\[
R\colon\mathsf{Var}\longrightarrow\mathcal{DM}_{gm,\acute{e}t}^{\eff}%
(\mathbb{C};\mathbb{Z})
\]
followed by $r_{X}$ of \eqref{lz3} is also a weakly $W$-exact functor.
It is clear that the induced map $K(\mathsf{Var})\rightarrow K({\MHS^{\eff}_{\Zb}})$ agrees with the map $DR$ in the square, where we first
go down and then to the right. \textit{(Step 3)} The commutativity of the
square in the claim amounts to showing that $RD$ and $DR$ are homotopic. By
Steps 1 and 2 both maps are induced from weakly $W$-exact functors and
therefore we only need to exhibit a natural transformation between these
functors,%
\[
R(-\times X)\Rightarrow R(-)\otimes R(X)\text{,}%
\]
which we had constructed in \eqref{l_KuennethProperty}.
\end{proof}

\begin{definition}\label{defi:loc-maps}
We define a map $\tilde{R}$ by means of the universal property of colimits
\[
\xymatrix{
K(\Var)_{\loc} \ar@{=}[d] \ar[r]^{\tilde{R}} & K(\MHS^{\eff}_{\Zb})_{\loc} \ar@{=}[d] \\
K(\Var)_{\loc}\ar@{=}[d] \ar[r]^-{\tilde{R}} &  K(\MHS^{\eff}_{\Zb})[R(\G_m)^{-1}] \ar@{=}[d] \\
\colim [K(\Var) \xrightarrow{\G_m \times ?} \cdots] \ar[r]^-R & \colim [K(\MHS^{\eff}_{\Zb}) \xrightarrow{R(\G_m) \otimes ?} \cdots].
}
\]
\end{definition}

\begin{definition}\label{lemma:wnfactors}
For $n \leq 0$ let $w_n\colon K(\MHS^{\eff}_{\Zb}) \to K(\Qb)$ denote the map sending a mixed Hodge structure to the underlying $\Qb$-vector space of its weight $n$ part. 


\begin{proof}
We need to justify why this map exists. Indeed, the functor%
\begin{equation}
\operatorname*{Gr}\nolimits_{n}^{W}\colon\ \MHS_{\mathbb{Z}%
}\longrightarrow\text{pure }\mathbb{Q}\text{-Hodge structures of weight
}n\label{lGrWExact}%
\end{equation}
is exact by \cite[Theorem 2.3.5, (iv)]{delignehodge2}. Precompose this with
the exact inclusion of a full subcategory $\MHS_{\mathbb{Z}}%
^{\eff}\rightarrow\ \MHS_{\mathbb{Z}}$, and postcompose it with the
exact functor forgetting the Hodge structure and just keeping the underlying
$\mathbb{Q}$-vector space. As exact functors induce maps in $K$-theory, the
desired map is constructed.
\end{proof}

\begin{rmk}
[\cite{MR2606234},\ Example 13.12]We recall that if $\mathcal{A}$ is an
abelian category and $\operatorname*{Fil}\mathcal{A}$ the exact category of
filtered objects in $\mathcal{A}$, the functor $\operatorname*{Gr}%
\nolimits_{n}\colon\operatorname*{Fil}\mathcal{A}\rightarrow\mathcal{A}$
usually fails to be exact. The above construction of $w_{n}$ relies on the
crucial fact that morphisms of mixed Hodge structures are automatically strict
(this also lies at the heart of the proof that $\MHS_{\mathbb{Z}}$ is
an abelian category and not merely exact).
\end{rmk}

\begin{rmk}
\label{rmk_AltDescriptionWn}There is an alternative construction of the maps
$w_{n}$. Denote by $\mathbf{S}_{n}$ the full subcategory of mixed Hodge
structures of weights $\leq n$. This is a Serre subcategory in $\MHS_{\mathbb{Z}}$.

Moreover, the tensor product satisfies%
\[
\mathbf{S}_{n}\times\mathbf{S}_{m}\longrightarrow\mathbf{S}_{n+m}\text{,}%
\]
so $\mathbf{S}_{0}$ is closed under the monoidal structure, contains the tensor unit, and for negative $n$, the $\mathbf{S}_{n}$ are tensor ideals in $\mathbf{S}_{0}$.

The exact functor of \eqref{lGrWExact} induces an exact
functor%
\[
\operatorname*{Gr}\nolimits_{n}^{W}\colon\mathbf{S}_{n}/\mathbf{S}%
_{n-1}\longrightarrow\text{pure }\mathbb{Q}\text{-Hodge structures of weight
}n\text{.}%
\]
This is an equivalence of abelian categories. An inverse functor can be
provided as follows: Choose a $\mathbb{Z}$-lattice $H_{\mathbb{Z}}$ in the
$\mathbb{Q}$-Hodge structure $H_{\mathbb{Q}}$. Then, $H_{\mathbb{Z}}$, equipped
with the same weight filtration of $H_{\mathbb{Q}}$ on $H_{\mathbb{Z}}%
\otimes\mathbb{Q}$, and the same Hodge filtration on $H_{\mathbb{Z}}%
\otimes\mathbb{C}$, lies in $\mathbf{S}_{n}$. If we had chosen a sublattice
$H_{\mathbb{Z}}^{\prime}\subseteq H_{\mathbb{Z}}$ instead, these lattices are
of finite-index and thus $H_{\mathbb{Z}}/H_{\mathbb{Z}}^{\prime}$ is a finite
abelian group (and thus canonically a mixed Hodge structure as $(H_{\mathbb{Z}%
}/H_{\mathbb{Z}}^{\prime})\otimes\mathbb{Q}=0$) and $H_{\mathbb{Z}%
}/H_{\mathbb{Z}}^{\prime}\in\mathbf{S}_{n-1}$, i.e. $H_{\mathbb{Z}}\cong
H_{\mathbb{Z}}^{\prime}$ in the quotient $\mathbf{S}_{n}/\mathbf{S}_{n-1}$. As
any two lattices have a common sublattice, this proves that this inverse
functor is well-defined. As $\MHS_{\mathbb{Z}}^{\eff}\cap\mathbf{S}%
_{n}=\MHS_{\mathbb{Z}}^{\eff}\cap\mathbf{S}_{0}$ for all $n\geq0$ by
Remark \ref{remark_NoPositiveWeightsInEffectiveMHS}, applying the localisation
sequence to this filtration, we may split off the $K$-theory of pure
$\mathbb{Q}$-Hodge structures of weights $0,-1,-2,$\ldots\ Composing with the
forgetful functor to $\mathbb{Q}$-vector spaces again yields the maps $w_{n}$.
\end{rmk}


\begin{lemma}\label{lemma_factoriz}
There exists a factorisation
\[
\xymatrix{
K(\MHS^{\eff}_{\Zb}) \ar[rd]^{w_n} \ar[d] & \\
K((\MHS^{\eff}_{\Zb}))_{\loc} \ar@{-->}[r] & K(\Qb).
}
\]
\end{lemma}
\begin{proof}  
As in Remark \ref{rmk_AltDescriptionWn}, denote by $\mathbf{S}_{n}$ the full subcategory of $\MHS^{\eff}_{\Zb}$, given by objects of weights $\leq n$. The quotient category $\MHS^{\eff}_{\Zb}/\mathbf{S}_{n-1}$ has $K_0$ where the class of $H_*^{\rm BM}(\G_m)=\Lb -1$ is invertible by a truncated geometric series.

Namely, for $m>\left\lceil \frac{n}{2}\right\rceil $%
\[
(1-\mathbb{L})(1+\mathbb{L}+\mathbb{L}^{2}+\cdots+\mathbb{L}^{m-1}%
)=1-\mathbb{L}^{m}=1
\]
since $\mathbb{L}^{m}$ has weight $-2m<n$.

The universal property of localisation yields a factorisation
\[
\xymatrix{
K(\MHS^{\eff}_{\Zb}) \ar[rd]^{w_n} \ar[d] & \\
K((\MHS^{\eff}_{\Zb})_{\loc}) \ar@{-->}[r] & K(\MHS^{\eff}_{\Zb}/\mathbf{S}_{n-1})\text{.}
}
\]
We conclude the proof by observing that $w_n$ factorises through $K(\MHS^{\eff}_{\Zb}/\mathbf{S}_{n-1})$, since $w_n$ vanishes on $K(\mathbf{S}_{n-1})$. The localisation theorem in $K$-theory, applied to the pair $\mathbf{S}_{n-1} \hookrightarrow \MHS^{\eff}_{\Zb} \twoheadrightarrow \MHS^{\eff}_{\Zb}/\mathbf{S}_{n-1}$, then yields the requisite factorisation $w_n\colon K(\MHS^{\eff}_{\Zb}/\mathbf{S}_{n-1}) \to K(\Qb)$. (See Remark \ref{rmk_AltDescriptionWn})
\end{proof}

By abuse of notation we will denote the factorisation above as well by
$$w_n \colon K(\MHS_{\Zb})_{\loc} \to K(\Qb).$$
\end{definition}

\subsection{Construction of the outgoing map}

Below we will use a plausible \emph{ansatz} to produce a retraction.


\begin{definition}\label{def_outgoing}
We define $r\colon K(\Var)_{\loc} \to K(\Qb)$ to be the composition $w_{-2} \circ \tilde{R}$.
\end{definition}

\begin{proposition}\label{prop_retract}
The map
\[
r \circ i|_{K(\Zb)^{\circ}_0} \colon K(\Zb)^{\circ}_0 \longrightarrow K(\Qb)^{\circ}_0
\]
induces the same maps on all homotopy groups $\pi_{\ast }$ as
\[
\varphi \colon K(\mathbb{Z})\longrightarrow K(\mathbb{Q})\qquad M\longmapsto M\otimes
\mathbb{Q},
\]
restricted to the connected component of $0$.
\end{proposition}
\begin{proof}
Recall that interpreting an integer matrix as a matrix with rational entries induces on the plus construction the same map as is induced from the exact functor of tensoring with $\Qb$ in terms of the $K$-theory of exact categories:
\[
\xymatrix{
\BGL(\Zb) \ar[r]^{\Zb \hookrightarrow \Qb } \ar[d]_k & \BGL(\Qb) \ar[d]^k \\
K(\Zb)^{\circ}_0 \ar[r]_{\varphi } & K(\Qb)^{\circ}_0 \text{.}
}
\]
Pick some $n>0$. Now consider the following diagram
\begin{equation}\label{eqn}
\xymatrix{
 & \BGL_n(\Zb) \ar[ld] \ar[d]^{\tilde{i}} \ar[r]^-{R} & \left({\MHS}^{\eff}_{\Zb}\right)^{\times} \ar[r]^-{w_{-2}} \ar[d] & F_f(\Qb)^{\times} \ar[d] \\
K(\Zb)_0^{\circ} \ar[r]^-i & K(\Var)_{\loc, \mathbf{1}} \ar[r]^-{\tilde{R}} & K({\MHS}^{\eff}_{\Zb})_{\loc, \mathbf{1}} \ar[r]^-{w_{-2}} & K(\Qb)^{\circ}_0
}
\end{equation}
in $\Ho(\Spaces)$.
The triangle on the left stems from Lemma \ref{lemma:comm3}, restricted to a specific $\BGL_n$.
The middle and right square come from Lemma \ref{lemma:comm1}. All downward arrows shift to the connected components of the direct sum neutral element $0$ (resp. the tensor units $\mathbf{1}$) as in \eqref{lcc1}, \eqref{lcc2}. This ensures that the diagrams for various $n$ are compatible under increasing $n$.

 The proof of the following claim will be provided subsequently.

\begin{claim}\label{claim}
The composition of the bottom row is a morphism of $H$-spaces with respect to direct sums.
\end{claim}

We have $H_2^{\rm BM}(\G_m,\Zb) = \Zb(2)$ and $H_1^{\rm BM}(\G_m,\Zb) = \Zb$. BM-homology of $\G_m$ vanishes in all other degrees.
Therefore,
\[
w_{-2}(H_*^{\rm BM}(\G_m^n,\Zb)) = \Qb^n,
\]
and an automorphism on the left, given by an element of $\GL_n(\Zb)$, is mapped to same matrix on the right, now interpreted as an element of $\GL_n(\Qb)$.

 We infer that the composition of the top row is naturally isomorphic to the map $\varphi$ induced from the functor $M\longmapsto M\otimes\mathbb{Q}$. We therefore obtain the following commutative diagram in $\Ho(\Spaces)$, where we denote the composition of the bottom row by $g$:
\[
\xymatrix{
\BGL_n(\Zb) \ar[r]^{\Zb \hookrightarrow \Qb } \ar[d]_c & \BGL_{n}(\Qb) \ar[d]^c \\
K(\Zb) \ar@{-->}[r]^g & K(\Qb)\text{.}
}
\]
This induces a homotopy commutative diagram
\[
\xymatrix{
\BGL(\Zb) \ar[r]^{\Zb \hookrightarrow \Qb } \ar[d]_k & \BGL(\Qb) \ar[d]^k \\
K(\Zb)^{\circ}_0 \ar@{-->}[r]^{g} & K(\Qb)^{\circ}_0 \text{.}
}
\]
Recall that $k\colon\BGL(\Zb) \to K(\Zb)^{\circ}=\BGL(\Zb)^+$ satisfies the almost-universal property of Theorem \ref{thm:magic}. It asserts that the graded map $\pi_*(g)$ is independent of the map of $H$-spaces making the diagram above homotopy commute. In particular, we have for our map above that $\pi_*(g)= \pi_*(\varphi)$.
This concludes the proof.

\end{proof}

%

\subsection{Proof of Claim \ref{claim}}

Let $(\Cc,\oplus,\otimes)$ be a unital symmetric \emph{bi-monoidal} category (also known as a \emph{rig category}). We will include the assumption that the neutral object $0$ is initial in $\Cc$.

The $K$-theory of a rig category $\Cc$, computed with respect to $\oplus$, inherits the structure of a ring spectrum. The ring operation is induced by $\otimes$, and it commutes up to homotopy. We will only use naive homotopy commutativity of this operation; $E_{\infty}$-structures will not be used multiplicatively.

It is clear that a symmetric bi-monoidal functor of rig categories $\Cc \to \Dc$ gives rise to a map of ring spectra
$$K^{\oplus}(\Cc) \to K^{\oplus}(\Dc).$$

\begin{definition}
\begin{itemize}
\item[(a)] We define the auxiliary rig category $\Tf$ as the smallest symmetric bi-monoidal full subcategory of $\MHS_{\Zb}^{\eff}$, which contains $0, \mathbf{1}, \Lb$, where we write $\Lb$ for $H_*^{\rm BM}(\Ab^1)$.
\item[(b)] Furthermore, we define $\Tf'$ to be
$\Tf / (\Tf\cap \mathbf{S}_{-2})$ (see the proof of Lemma \ref{lemma:wnfactors} for a definition of $\mathbf{S}_{-2}$).
\end{itemize}
\end{definition}

By virtue of definition, $K(\Tf')$ equivalent to $K(\Zb) \oplus K(\Zb)$. We denote the projection onto the first factor by $p_0$. And, $K_0(\Tf')$ is isomorphic to the free abelian group generated by $1$ and $\Lb$. 

\begin{definition}\label{def:bip-sw}
Let $X \in \Var$ be an arbitrary variety. We denote by $\bip_{\Var}(X)$ the smallest (non-full) subcategory of $\Var$, which inherits the structure of a \emph{split} subtractive category. That is, a morphism of varieties belongs to $\bip_{\Var}(X)$ if and only if it is a closed and open immersion. Furthermore, we say that $A \xrightarrow{i} B \xleftarrow{j} C$ is a subtraction sequence, if 
\[
\xymatrix{
\emptyset \ar[r] \ar[d] & A \ar[d]^i \\
C \ar[r]^j & B
}
\] 
is a pushout square.
\end{definition}

\begin{lemma}\label{lemma:bip-sw}
Definition \ref{def:bip-sw} endows $\bip_{\Var}(X)$ with the structure of a (split) subtractive category. The functor $\bip_{\Var}(X) \to \Var$ is exact.
\end{lemma}
\begin{proof}
The objects of $\bip_{\Var}(X)$ are given by $\emptyset, {\rm pt}, X, X^2, \cdots$ and finite disjoint unions thereof. By virtue of definition, every morphism in $\bip_{\Var}(X)$ can be written as $f\sqcup i_Y$, where $f$ is an isomorphism and $i_Y\colon \emptyset \to Y$ is the canonical map. Thus, a morphism in $\bip_{\Var}(X)$ is an open-closed immersion with complement in $\bip_{\Var}(X)$. The assertion now follows from the fact that $\Var$ is subtractive and that the aforementioned class of morphisms is preserved by base change in $\bip_{\Var}(X)$.
\end{proof}

\begin{lemma}\label{lemma:Kplus-sub}
Let $\C$ be the \emph{split} subtractive category, associated to the symmetric monoidal structure $\sqcup$ given by finite coproducts in $\C$. Then, $K^{\sqcup}(\C)$ is equivalent to $K$-theory of the subtractive category $\C$.
\end{lemma}
\begin{proof}
One can associate a split Waldhausen category structure to $\sqcup$ on $\C$. It is clear from the definitions that the corresponding Waldhausen $S$-construction is equivalent (as a simplicial category) to Campbell's $S$-construction. Passing to loop spaces of geometric realisations, we obtain a proof of the assertion, since Waldhausen $K$-theory agrees of $\C$ agrees with $K^{\sqcup}(\C)$, by virtue of \cite[Theorem 1.8.1]{wald}.
\end{proof}

The K\"unneth formula implies that there are symmetric functors of bi-permutative categories
$$\bip_{\Var}(\G_m) \to \Tf \text{ and } \Tf \to \Tf'\text{.}$$
We therefore obtain a morphisms of ring spectra $R'\colon K(\bip_{\Var}(\G_m)) \to K(\Tf)$.

\begin{lemma}\label{lemma:bip-L}
The morphism $R'$ fits into a homotopy commutative diagram \emph{of spectra} (forgetting the ring structure)
\begin{equation}\label{eqn:bip-L}
\xymatrix{
K(\bip_{\Var}(\G_m)) \ar[r]^-{R'}\ar[d] & K(\Tf) \ar[d]  \\
K(\Var) \ar[r]^-{\tilde{R}} & K(\MHS^{\eff}_{\Zb}) \text{.}
}
\end{equation}
\end{lemma}
\begin{proof}
There is a commuting square of symmetric monoidal categories
\[
\xymatrix{
(\bip_{\Var}(\G_m),\sqcup) \ar[r]\ar[d] & (\Tf,\oplus) \ar[d]  \\
(\Var,\sqcup) \ar[r] & (\MHS^{\eff}_{\Zb},\oplus) \text{.}
}
\]
Applying $K^{\oplus}$, we obtain a homotopy commutative diagram
\begin{equation}\label{eqn:Kplus}
\xymatrix{
K^{\sqcup}(\bip_{\Var}(\G_m)) \ar[r]\ar[d] & K^{\oplus}(\Tf) \ar[d]  \\
K^{\sqcup}(\Var) \ar[r] & K^{\oplus}(\MHS^{\eff}_{\Zb}) \text{.}
}
\end{equation}
The symmetric monoidal functor $\Var \to \MHS^{\eff}_{\Zb}$ can be lifted to a W-exact functor $R$ from $\Var \to \MHS^{\eff}_{\Zb}$. We therefore have a commutative diagram of categories
\[
\xymatrix{
(\Var,\sqcup) \ar[r] \ar[d] & (\MHS^{\eff}_{\Zb},\oplus) \ar[d] \\
\Var \ar[r]^{R} & \MHS^{\eff}_{\Zb}\text{,}
}
\]
where the rows are $W$-exact functors, categories on the left-hand side are subtractive, categories on the right-hand side are Waldhausen, and the symmetric monoidal categories of the top row are seen as split subtractive or split Waldhausen categories.

This induces a homotopy commuting square of $K$-theory spaces
\[
\xymatrix{
K^{\sqcup}(\Var) \ar[r] \ar[d] & K^{\oplus}(\MHS^{\eff}_{\Zb}) \ar[d] \\
K(\Var) \ar[r]^-{R} & K(\MHS^{\eff}_{\Zb})\text{,}
}
\]
where we use the identification from Lemma \ref{lemma:Kplus-sub}. Appending this diagram to \eqref{eqn:Kplus} yields
\[
\xymatrix{
K^{\sqcup}(\bip_{\Var}(\G_m)) \ar[r]\ar[d] & K^{\oplus}(\Tf) \ar[d]  \\
K^{\sqcup}(\Var) \ar[r] \ar[d] & K^{\oplus}(\MHS^{\eff}_{\Zb}) \ar[d] \\
K(\Var) \ar[r]^{R} & K(\MHS^{\eff}_{\Zb})\text{.}
}
\]
We obtain the requisite homotopy commutative diagram by omitting the middle row and applying the identification from Lemma \ref{lemma:Kplus-sub}.
\end{proof}

\begin{definition}\label{defi:GL11}
We define $\GL_1(K(\Tf'))_{+1}$ by the following cartesian diagram in the $\infty$-category of spaces:
\[
\xymatrix{
\GL_1(K(\Tf'))_{+1} \ar[r] \ar[d] & \{\mathbf{1}\} \ar[d] & \\
\GL_1(K(\Tf')) \ar[r]^{p_0} & \GL_1(K(\Zb)) \text{.}
}
\]
\end{definition}

\begin{lemma}\label{lemma:w}
The map $w_{-2}\colon \GL_1(K(\Tf'))_{+1} \to K(\Zb)$ is a morphism of $H$-spaces intertwining $\otimes$ and $\oplus$.
\end{lemma}
\begin{proof}
We begin the proof by computing the homotopy groups of $\GL_1(K(\Tf'))_{+1}$. By virtue of Definition \ref{defi:GL11} and fibration sequence, $\pi_*(\GL_1(K(\Tf'))_{+1}) \simeq \pi_*(K(\Zb))$.

The map $x \mapsto (\mathbf{1} \oplus x(2))$ is a morphism of $H$-spaces $\exp\colon K(\Zb)\to \GL_1(K(\Tf'))_{+1}$, inducing an isomorphism on homotopy groups. By Whitehead's lemma it is therefore a homotopy equivalence.

We conclude that the homotopy inverse $\exp^{-1}$ is likewise a map of $H$-spaces. The commutative diagram
\[
\xymatrix{
K(\Zb) \ar[r]^-{\exp} \ar[rd]_{\mathrm{can}} & \GL_1(K(\Tf'))_{+1} \ar[d]^{w_{-2}} \\
& K(\Qb)
}
\]
implies that $w_{-2} = \mathrm{can} \circ \exp^{-1}$ is a map of $H$-spaces.
This concludes the proof.
\end{proof}

\begin{proof}[Proof of Claim \ref{claim}]
Consider the following homotopy commutative diagram:
\[
\xymatrix{
K(\Zb)^{\circ}_0\ar[rd]_i \ar[r]^-{i'}  & \GL_1(K(\bip_{\Var}(\G_m)_{\loc}))_{+1} \ar[d]^{z} \ar[r]^-{R'} & \GL_1(K(\Tf))_{+1} \ar[r]^-{Q} \ar[d] & \GL_1(K(\Tf'))_{+1} \ar[d]^-{w_{-2}} \\
 & \GL_1(K(\Var)_{\loc}) \ar[r]^-{\tilde{R}} & \GL_1(K({\MHS}^{\eff}_{\Zb})_{\loc}) \ar[r]^-{w_{-2}} & K(\Qb)\text{.}
}
\]
The middle square is obtained from the $\G_m$-localisation of \eqref{eqn:bip-L}, and the right-hand commuting square is obtained (up to sign), by applying $K$-theory to the following commutative diagram of exact categories:
\[
\xymatrix{
\Tf \ar[r] \ar[d] & \Tf' \ar[d]^{w_{-2}} \\
\MHS_{\Zb}^{\eff} \ar[r]^{w_{-2}} & F_f(\Qb)\text{.}
}
\]
We showed in Lemma \ref{lemma:w} that $w_{-2}\colon \GL_1(K(\Tf'))_{+1} \to K(\Qb)$ is an $H$-space morphism, intertwining $\otimes$ with $\oplus$. Likewise, we know from Lemma \ref{lemma:bip-L} that $R'$ is a morphism of ring spectra (therefore, intertwines $\oplus$ with $\oplus$ and $\otimes$ with $\otimes$). The map $Q$ satisfies the same property and $i'$ is a morphism of $H$-spaces intertwining $\oplus$ with $\otimes$. We conclude the proof by observing that the composition 
$$w_{-2} \circ Q \circ R' \circ i'$$
is an $H$-space-map intertwining $\oplus$ with $\otimes$. 
\end{proof}

\section{Proof of the first theorem}

Write $L_{\mathbb{Q}}$ for the localisation from the unstable homotopy
category (say for simply connected spaces) to the rational homotopy category.

\begin{theorem}
\label{thm1}  For $\Var$ referring to varieties over the complex numbers, the following holds:

\begin{enumerate}
\item For every $n\geq1$ the group $\pi_{n}\left(  K(\mathsf{Var})
[\mathbb{G}_{m}^{-1}]\right)$ has the Quillen $K$-group $K_{n}(\mathbb{Z})$ as a
direct summand.

\item The rational homotopy type $L_{\mathbb{Q}}K(\mathbb{Z})$ is a retract of
$L_{\mathbb{Q}}\left(  K(\mathsf{Var})[\mathbb{G}_{m}^{-1}]\right)  $ after
going to the $1$-connective cover.
\end{enumerate}
\end{theorem}

\begin{proof}
(1) The map $K_{n}(\mathbb{Z})\rightarrow K_{n}(\mathbb{Q})$ induced from
$M\mapsto M\otimes\mathbb{Q}$ is injective for all $n$ (for the details why
this holds, we refer to the more general proof of Theorem \ref{thm2} below).
We may therefore interpret $K_{n}(\mathbb{Z})$ as a subgroup of $K_{n}%
(\mathbb{Q})$. By Proposition \ref{prop_retract} it follows for $n\geq1$ that%
\[
K_{n}(\mathbb{Z})\overset{i}{\longrightarrow}{\pi}_{n}K(\mathsf{Var}[\mathbb{G}%
_{m}^{-1}])\overset{r}{\longrightarrow}K_{n}(\mathbb{Q})
\]
is injective and has image $K_{n}(\mathbb{Z})$ with $r\circ
i=\operatorname*{id}_{K_{n}(\mathbb{Z})}$. This exhibits a splitting to $i$,
proving the claim. For (2) we use that%
\[
K_{n}(\mathbb{Z})\otimes\mathbb{Q}\underset{\sim}{\overset{(\ast
)}{\longrightarrow}}K_{n}(\mathbb{Q})\otimes\mathbb{Q}%
\]
is an isomorphism for all $n\geq2$ (\cite[Chapter IV, Theorem 1.17]{K-book}, or see the proof of Theorem \ref{thm2} for details why
this holds). Hence, for the $1$-connective cover we obtain by Whitehead's
lemma that the arrow $r\circ i$ induces an equivalence $L_{\mathbb{Q}%
}K(\mathbb{Z})\rightarrow L_{\mathbb{Q}}K(\mathbb{Q})$ of homotopy types.
Thus, the map $r\circ i\mid_{K(\mathbb{Z})_{0}^{\circ}}$ of Proposition
\ref{prop_retract} exhibits the rationalisation $L_{\mathbb{Q}}\left(
K(\mathsf{Var})[\mathbb{G}_{m}^{-1}]\right)  $ as a retract in the rational
homotopy category
\end{proof}


\begin{corollary}\label{cor}
For every positive integer $n$, there exists an element of infinite order in $K_{4n+1}(\Var)$.
\end{corollary}
\begin{proof}
According to Borel, the $\Zb$-module $K_{4n+1}(\Zb)$ has rank $1$ (see \cite[Proposition 12.2]{borel} or \cite[Chapter IV, Theorem 1.18]{K-book}). We infer from the preceding result that $K_{4n+1}(\Var)_{\loc}$ contains an element of infinite order.

Since $K_n(\Var)_{\loc}$ is given by the localisation of the $K_0(\Var)$-module $K_n(\Var)$ in $[\G_m] \in K_0(\Var)$ (see \cite[Proposition 7.2.3.25]{ha} and \cite[Example  7.5.0.7]{ha}), we see that there must exist an element of infinite order in $K_n(\Var)$.
\end{proof}

There is a natural morphism
$$\pi_n(\Sb) \otimes K_0(\Var) \to K_n(\Var),$$
see \cite[Section 7]{cwz}. Elements in the image are called \emph{permutative} or \emph{decomposable}.

For finite fields $k=\mathbb{F}_q$ non-permutative elements have been constructed for $K_1(\Var)$ by Zakharevich \cite{zaktalk}. These classes are also torsion in the $K$-group.

The following corollary provides an affirmative answer to \cite[Question 7.1]{cwz} for the field of complex numbers.

\begin{corollary}
For every positive integer $n$, there exists an \emph{indecomposable} (i.e. non-permutative) element in $K_{4n+1}(\Var)$.
\end{corollary}
\begin{proof}
The stable homotopy groups $\pi_n(\Sb)$ are torsion if $n > 0$. This implies that a decomposable element in $K_n(\Var)$ is torsion, whenever $n$ is positive. Hence, Corollary \ref{cor} implies the existence of an indecomposable class in $K_{4n+1}(\Var)$.
\end{proof}


\section{Retracts from abelian varieties}

\subsection{What changes?}

In \S \ref{section_torusconstruction} we have constructed a pseudo-retraction%
\[
K(\mathbb{Z})_{0}^{\circ}\longrightarrow K(\mathsf{Var}_{\mathsf{loc}%
})_{\mathbf{1}}^{\circ}\longrightarrow K(\mathbb{Q})^{\circ}_0
\]
based on split tori and for the field $k=\mathbb{C}$. In this section we
present a more elaborate version of the same construction. We replace $\Cb$ by a perfect field,
$\mathbb{G}_{m}$ by an abelian variety and the mixed Hodge realisation by
$1$-isomotives. This complicates the technical details, yet the homotopical
constructions remain the same, mutatis mutandis.

Another slight difference is that when working with the torus, weights $0$ and $-2$ played the vital role, while for abelian varieties the weights $0$ and $-1$ become important and we shall entirely discard weight $-2$ information. Presumably one could set up a uniform construction for semi-abelian varieties, but this seems more involved because we would have to keep track of weights $0,-1,-2$ simultaneously and the added generality does not give us anything we cannot obtain by dealing with both cases separately.

\subsection{Construction of the incoming map}

Let $k$ be a perfect field. Let $A$ be an abelian variety over $k$. Let $\mathcal{O}%
:=\operatorname*{End}(A)$ be the endomorphism ring of $A$ as a commutative
group variety. Define%
\[
F:=\mathcal{O}\otimes_{\mathbb{Z}}\mathbb{Q}\text{.}%
\]
Then, $F$ is a finite-dimensional semisimple $\mathbb{Q}$-algebra and
$\mathcal{O}\subset F$ a ($\Zb$-)order. Write $\mathsf{APow}\subset\mathsf{Var}$ for the full subcategory of powers of
$A$ (that is: the category of varieties which admit an isomorphism to $A^{n}$
for some $n\geq0$).

There is a symmetric monoidal functor%
\[
(F_{f}(\mathcal{O}),\oplus)^{\times}\longrightarrow(\mathsf{APow}%
,\times)^{\times}%
\]
sending $\mathcal{O}$ to $A$.

\begin{rmk}
Entirely analogously to the situation of Remark \ref{rmk_WhatAreTheMapsForGm},
this is a faithful functor and the morphisms in its image are those preserving
the neutral element of the abelian variety.
\end{rmk}

We define%
\[
\mathsf{Var}_{\mathsf{loc}}:=\mathsf{Var}[A^{-1}]\text{.}%
\]
We obtain the same diagram (\ref{lcc1}) with $\mathbb{Z}$ replaced by
$\mathcal{O}$ and $\mathbb{G}_{m}$ by $A$. Exactly as for \eqref{lcc1} we obtain a diagram

\begin{equation}\label{lcc1_av_version}
\xymatrix{
\BGL_n(\Oo) \ar[r] \ar[d] & \BGL_{n+1}(\Oo) \ar[d] \\
\big(K(\Var)_{\loc}\big)^{\circ}_{A^n} \ar[r]^-{\times A}_{\sim } & \big(K(\Var)_{\loc}\big)^{\circ}_{A^{n+1}}
}
\end{equation}
The analogue of Lemma
\ref{lemma:comm3} is proven in exactly the same way.

\begin{lemma}\label{lemma:comm3_av}
There exists a morphism $i\colon K(\Oo)^{\circ}_0 \to (K(\Var)_{\loc},\times)$ of $H$-spaces such that the following diagram
\[
\xymatrix{
\BGL(\Oo) \ar[rd]^{\tilde{i}} \ar[d]_{c} &  \\
K(\Oo)^{\circ}_0 \ar@{-->}[r]^-i & \big(K(\Var)_{\loc}\big)^{\circ}_{\mathbf{1}}
}
\]
commutes.
\end{lemma}

\subsection{Constructions from the theory of Deligne \texorpdfstring{$1$}{1}-motives}

We recall the most important ingredients from the theory of $1$-motives. Our
primary source is \cite{bv}. All our group schemes will be assumed commutative
and over $k$. Let $\mathsf{GrpSchm}$ be the category of commutative group
scheme of \textit{locally} finite type. This is an abelian category by
\cite[Lemma 1.1.1]{bv}. By a \emph{quasi-lattice} $L$ we mean a group
scheme which \'{e}tale locally is isomorphic to a finitely generated abelian
constant group scheme. Torsion is permitted. Write $\mathsf{AV}\subset
\mathsf{GrpSchm}$ for the full subcategory of abelian varieties.

A $1$\emph{-motive with torsion} (\cite[\S 1.9]{bv}) is a complex of group
schemes%
\[
M=[L\overset{u}{\longrightarrow}G]\text{,}%
\]
where $L$ is a quasi-lattice and $G$ semi-abelian. We may regard this as a
full subcategory of the category of bounded complexes of $\mathsf{GrpSchm}$.
We write ${\mathsf{Mot}}_{1}^{\operatorname*{tor}}$ for the category of
$1$-motives with torsion. Its objects are $1$-motives with torsion and
morphisms are%
\begin{equation}
\operatorname*{Hom}\nolimits_{{\mathsf{Mot}}_{1}^{\operatorname*{tor}}%
}(M,M^{\prime}):=\operatorname*{Hom}\nolimits_{\operatorname*{complex}%
}(M,M^{\prime})[qis^{-1}]\text{,} \label{eq_l1}%
\end{equation}
where $\operatorname*{Hom}\nolimits_{\operatorname*{complex}}(M,M^{\prime})$
denotes morphisms between the underlying complexes of $k$-group schemes, and
$qis$ the family of quasi-isomorphisms of complexes. Being a $2$-term complex,
this amounts to inducing an isomorphism on the respective $\ker(u)$ and
$\operatorname*{coker}(u)$ of $[L\overset{u}{\longrightarrow}G]$. See
\cite[C.3.1]{bv} for details.

\begin{rmk}
\label{rmk_IdentifyQis}It is useful to recall an explicit description of the
quasi-isomorphisms: They amount precisely to the morphisms $f:M\rightarrow
M^{\prime}$ of complexes%
\[
f:[L\overset{u}{\longrightarrow}G]\longrightarrow\lbrack L^{\prime}%
\overset{u^{\prime}}{\longrightarrow}G^{\prime}]
\]
which induce (1) an isomorphism $\ker(L\rightarrow L^{\prime})\cong%
\ker(G\rightarrow G^{\prime})$, and (2) such that these are both finite
\'{e}tale group schemes. The proof for this is given as \cite[C.2.2 (2)]{bv}.
\end{rmk}

Every $1$-motive $M$ comes with its canonical increasing weight filtration:
Write $M=[L\overset{u}{\longrightarrow}G]$ and let%
\[
W_{i}(M):=\left\{
\begin{array}
[c]{ll}%
M & \text{for }i\geq0\\
G & \text{for }i=-1\\
T & \text{for }i=-2\\
0 & \text{for }i\leq-3\text{,}%
\end{array}
\right.
\]
where $T\hookrightarrow G\twoheadrightarrow A$ is the presentation of the
semi-abelian variety $G$ as an extension of an abelian variety by a torus $T$.
Moreover,%
\begin{equation}
\operatorname*{Gr}\nolimits_{-2}^{W}(M)=[0\longrightarrow T]\text{,}%
\qquad\operatorname*{Gr}\nolimits_{-1}^{W}(M)=[0\longrightarrow A]\text{,}%
\qquad\operatorname*{Gr}\nolimits_{0}^{W}(M)=[L\longrightarrow0]\text{.}
\label{lwax0}%
\end{equation}

We write ${\mathsf{Mot}}_{1}^{\operatorname*{iso}}$ for the category
${\mathsf{Mot}}_{1}^{\operatorname*{tor}}\otimes\mathbb{Q}$, i.e. the category
with the same objects, but its homomorphism groups get tensored with
$\mathbb{Q}$. The category ${\mathsf{Mot}}_{1}^{\operatorname*{iso}}$ is
abelian (\cite[C.5.3]{bv} or \cite[Lemma 3.2.2]{MR2102056}). It also carries a
symmetric monoidal product%
\[
\otimes_{1}\colon{\mathsf{Mot}}_{1}^{\operatorname*{iso}}\times{\mathsf{Mot}%
}_{1}^{\operatorname*{iso}}\longrightarrow{\mathsf{Mot}}_{1}%
^{\operatorname*{iso}}%
\]
described in \cite[Proposition 7.1.2]{bv}, which respects the weight filtration and
is bi-exact, see loc. cit.

Analogously, it has a symmetric monoidal full subcategory $\mathsf{AV}%
^{\operatorname*{iso}}$ of abelian varieties up to isogeny.

\begin{rmk}
\label{rmk_VersionsOfOneMotiveCategories}Many variations are possible. In
characteristic zero the category ${\mathsf{Mot}}_{1}^{\operatorname*{tor}}$ is
itself abelian (\cite[C.5.3]{bv}). In characteristic $p>0$ the category
${\mathsf{Mot}}_{1}^{\operatorname*{tor}}\otimes\mathbb{Z}[\frac{1}{p}]$ is
abelian. Without inverting $p$, the finite group scheme kernel of the isogeny
$[p]\colon A\rightarrow A$ need not be \'{e}tale. One can also study
$1$-motives without torsion, denoted by ${\mathsf{Mot}}_{1}$. This is not an
abelian category, but after tensoring with $\mathbb{Q}$ also gives the same
category ${\mathsf{Mot}}_{1}^{\operatorname*{iso}}$.
\end{rmk}

\begin{lemma}
\label{lemma_KThyOfAVIsogenyCategory}$K(\mathsf{AV}^{\operatorname*{iso}%
})\cong\bigoplus_{i\in I}K(\operatorname*{End}\nolimits_{k}A_{i})$, where $I$
is the set of isogeny classes of simple abelian varieties over $k$ and $A_{i}$
a representative of the isogeny class. Moreover, $\operatorname*{End}%
\nolimits_{k}A_{i}$ is a finite-dimensional division algebra over $\mathbb{Q}$.
\end{lemma}

\begin{proof}
By Poincar\'{e} reducibility the category $\mathsf{AV}^{\operatorname*{iso}}$
is abelian semisimple. Thus, there is an equivalence of categories%
\[
\mathsf{AV}^{\operatorname*{iso}}\overset{\sim}{\longrightarrow}%
\bigoplus_{i\in I}P_{f}(\operatorname*{End}\nolimits_{k}A_{i})\text{,}%
\]
where $P_{f}$ denotes the category of finitely generated
projective\footnote{As the endomorphism rings are division algebras by Schur's
Lemma, all modules are projective.} modules and $I$ is the set of isomorphism
classes of the simple objects of the category. Finally, $K(\mathcal{C}%
\times\mathcal{C}^{\prime})\cong K(\mathcal{C})\oplus K(\mathcal{C}^{\prime})$
for arbitrary $\mathcal{C},\mathcal{C}^{\prime}$ and $K$-theory commutes with
filtering colimits.
\end{proof}

Let $\mathsf{\tilde{T}}_{\mathbb{Z}}$ be the category of $2$-term complexes%
\[
\lbrack L\overset{0}{\longrightarrow}A]\text{,}%
\]
concentrated in degrees $[0,1]$, with $L$ a quasi-lattice, $A$ an abelian
variety and zero differential. We may regard this as a complex in the abelian
category $\mathsf{GrpSchm}$. The category $\mathsf{GrpSchm}$ is also symmetric
monoidal as the product of group schemes is again a group scheme.

Define $\mathsf{\tilde{T}}:=\mathsf{\tilde{T}}_{\mathbb{Z}}\otimes\mathbb{Q}$.
The category $\mathsf{\tilde{T}}$ has a symmetric monoidal structure coming
from the multiplication of complexes:%
\begin{equation}
\lbrack L_{1}\overset{0}{\longrightarrow}A_{1}]\otimes\lbrack L_{2}\overset
{0}{\longrightarrow}A_{2}]=\tau_{\leq1}\operatorname*{Total}\left(
\begin{array}
[c]{ccc}%
L_{1}\otimes A_{2} & \longrightarrow & A_{1}\otimes A_{2}\\
\uparrow &  & \uparrow\\
L_{1}\otimes L_{2} & \longrightarrow & L_{2}\otimes A_{1}%
\end{array}
\right)  \text{,}\label{lwax3}%
\end{equation}
i.e. we truncate the total complex of the double complex to degrees $[0,1]$.
This corresponds to removing the entry $A_{1}\otimes A_{2}$ in the double
complex. (All morphisms in the double complex on the right are zero maps).

If $(X,x)$ is a variety and $x\in X(k)$ a $k$-rational point, one may consider
the collection of maps%
\[
\left\{  (A,f)\mid f\colon(X,x)\longrightarrow(A,0)\right\}  \text{,}%
\]
where $A$ is an abelian variety and $f$ a morphism mapping the distinguished
point $x$ to the neutral element of $A$. The \emph{Albanese}
$\operatorname*{Alb}(X,x)$ is the initial object in this family.

\begin{lemma}
\label{lemma_Psi}The functor%
\begin{align}
\Psi\colon({\mathsf{Mot}}_{1}^{\operatorname*{iso}},\otimes_{1}) &
\longrightarrow(\mathsf{\tilde{T},\otimes})\nonumber\\
\lbrack L\overset{u}{\longrightarrow}G] &  \longmapsto\lbrack L\overset
{0}{\longrightarrow}\operatorname*{Alb}(G,1_{G})]\text{,}\label{l_v2}%
\end{align}
is exact and symmetric monoidal.
\end{lemma}

The functor can be understood as truncating to weights $0$ and $-1$.

\begin{proof}
Equation \eqref{l_v2} defines a functor ${\mathsf{Mot}}_{1}^{\operatorname*{tor}%
}\rightarrow\mathsf{\tilde{T}}$ (the equivalence relation in \eqref{eq_l1} for homomorphisms poses no issue as quasi-isomorphisms are surjective with a finite group scheme as their kernel by Remark \ref{rmk_IdentifyQis} and thus genuine isomorphisms after tensoring the category $\mathsf{\tilde{T}}_{\Zb}$ with $\Qb$). The Albanese functor
$G\mapsto\operatorname*{Alb}(G,1_{G})$ is exact up to isogeny, \cite[Lemma
4.7, Proposition 4.8]{MR3650225}. It remains to settle monoidality. The natural transformations which are part of the datum of a monoidal functor, e.g. \eqref{l_nattransfstrongmonoidal}, stem mostly from the fact that we only truncate the usual product of complexes, along with a little side discussion because of the built-in Albanese. The tensor
product $\otimes_{1}$ on ${\mathsf{Mot}}_{1}^{\operatorname*{iso}}$ is set up
in \cite[Proposition 7.1.2]{bv}. By (c) \textit{loc. cit.} the product of
$[L_{1}\overset{u_{1}}{\longrightarrow}G_{1}]\otimes\lbrack L_{2}%
\overset{u_{2}}{\longrightarrow}G_{2}]$ is given by%
\[
\lbrack L_{1}\otimes L_{2}\overset{\tilde{u}}{\longrightarrow}G]\text{,}%
\]
where%
\begin{equation}
\operatorname*{Biext}(G_{1}\otimes^{\mathbb{L}}G_{2},\mathbb{G}_{m})^{\ast
}\hookrightarrow G\twoheadrightarrow\left(  L_{1}\otimes G_{2}\right)
\oplus\left(  L_{2}\otimes G_{1}\right)  \text{,}\label{lwax2a}%
\end{equation}
where $(-)^{\ast}$ is the generalised Cartier dual for $1$-motives\footnote{It
restricts to the ordinary Cartier dual on tori or quasi-lattices, but for
example an abelian variety goes to its dual abelian variety (whereas the
na\"{\i}ve Cartier dual of an abelian variety would be zero).}. We note that
the $G_{i}$ are in weights $\{-1,-2\}$, and since the tensor product respects
weights by (d) \textit{loc. cit.}, the term $\operatorname*{Biext}(G_{1}%
\otimes^{\mathbb{L}}G_{2},\mathbb{G}_{m})^{\ast}$ sits in weight $-2$ for
weight reasons, (d) also settles the exactness. We therefore see that
\eqref{l_v2} sends this to the tensor product of \eqref{lwax3}.
\end{proof}

Using \cite{paperletto} we obtain a realisation functor to $1$-motives%
\[
R_{1}\colon\mathsf{Var}\longrightarrow{\mathsf{Mot}}_{1}\text{.}%
\]
We only need the simpler version with $\mathbb{Q}$-coefficients even though we remark that 
\cite{bv} and \cite[\S 2.11]{MR3004172} discuss the integral version as well. This is
a weakly $W$-exact functor by the constructions in \cite{paperletto}.

\begin{rmk}
The functor underlying the realisation to $1$-motives is a localisation
functor. Concretely, it can also be understood as the left derived functor of
the motivic Albanese (\cite{bv, MR2494373}).
\end{rmk}

Based on the above realisation, we build a simpler weakly $W$-exact functor%
\[
R\colon\mathsf{Var}\overset{R_{1}}{\longrightarrow}{\mathsf{Mot}}%
_{1}^{\operatorname*{iso}}\overset{\Psi}{\longrightarrow}\mathsf{\tilde{T}}%
\]
by composing the realisation to $1$-motives, tensored with $\mathbb{Q}$, and
postcomposing with the exact functor $\Psi$.

Lemma \ref{lemma_commute_realiz} is easy to adapt.

\begin{lemma}
For every $X \in \Var$ we have a commutative diagram 
\[
\xymatrix{
K(\Var) \ar[r]^{X \times ? } \ar[d] & K(\Var) \ar[d] \\
K({\widetilde{\mathsf{T}}}) \ar[r]^{R(X) \otimes ?} & K({\widetilde{\mathsf{T}}})
}
\]
in the homotopy category of spaces.
\end{lemma}

\begin{proof}
The beginning of the proof is exactly the same for Lemma \ref{lemma_commute_realiz}, just replace $\MHS_{\Zb}^{\eff}$ with $\widetilde{\mathsf{T}}$. With this
modification, we reduce to having to exhibit a natural transformation of
functors%
\[
R(-\times X)\Rightarrow R(-)\otimes R(X)\text{.}%
\]
We do this in two steps: As $R=\Psi\circ R_{1}$, we first use the\ K\"{u}nneth
property of the realisation $R_{1}$ to $1$-isomotives. This amounts to showing
that $\mathcal{DM}_{gm}^{\eff}\rightarrow\mathsf{Mot}_{1}^{\operatorname*{iso}%
}$ respects the tensor product (\cite[Proposition 7.1.2., (a)]{bv}), and that $\Psi$
respects the tensor product (Lemma \ref{lemma_Psi}).

In the situation of positive characteristic (and only then) we should point out that our discussion of the strong monoidal functor from varieties to mixed motives found around \eqref{l_DiscussCDHDescentExtension} required working with motives with coefficients in $\mathbb{Z}[\frac{1}{p}]$-coefficients. Since $1$-isomotives and $\widetilde{\mathsf{T}}$ have $\Qb$-coefficients, this is fine in the situation at hand.
\end{proof}

We adapt Definition \ref{defi:loc-maps}, replacing the torus by the fixed abelian variety $A$.

\begin{definition}
We define a map $\tilde{R}$ by means of the universal property of colimits
\[
\xymatrix{
K(\Var)_{\loc} \ar@{=}[d] \ar[r]^{\tilde{R}} & K({\widetilde{\mathsf{T}}})_{\loc} \ar@{=}[d] \\
K(\Var)_{\loc}\ar@{=}[d] \ar[r]^-{\tilde{R}} &  K({\widetilde{\mathsf{T}}})[R(A)^{-1}] \ar@{=}[d] \\
\colim [K(\Var) \xrightarrow{A \times ?} \cdots] \ar[r]^-R & \colim [K({\widetilde{\mathsf{T}}}) \xrightarrow{R(A) \otimes ?} \cdots].
}
\]
\end{definition}

Note that our notation $K(\Var)_{\loc}$ suppresses the choice of $A$.

\begin{definition}
Define%
\begin{align*}
w_{-1}\colon K(\mathsf{\tilde{T}}) &  \longrightarrow K(\mathsf{AV}%
^{\operatorname*{iso}})\\
M &  \longmapsto\operatorname*{Gr}\nolimits_{-1}^{W}(M)
\end{align*}
by sending an object to its weight $-1$ piece.
\end{definition}

\begin{proof}
The functor $[L\overset{0}{\longrightarrow}A]\mapsto A$ is clearly exact and
thus induces a map on $K$-theory.
\end{proof}

Next we deduce the immediate analogue of Lemma \ref{lemma_factoriz}.

\begin{lemma}
There exists a factorisation
\[
\xymatrix{
K({\mathsf{\tilde{T}}}) \ar[rd]^{w_{-1}} \ar[d] & \\
K({\mathsf{\tilde{T}}})_{\loc} \ar@{-->}[r] & K(\mathsf{AV}^{\operatorname{iso}}).
}
\]
\end{lemma}

\begin{proof}
The computation of the BM homology of $\Gb_m$ gets replaced by the computation of the realisation of $A$ in ${\mathsf{\tilde{T}}}$: $R(A)=[\mathbb{Z}\overset{0}{\longrightarrow}A]$. If $[A]\colon K(\mathsf{\tilde{T}})\rightarrow K(\mathsf{\tilde{T}})$ denotes
the multiplication with $A$ (this is exact), we get a well-defined homotopy
class of a map $-[A]$ since $K(\mathsf{\tilde{T}})$ is a spectrum. In $K_0$ we compute%
\[
(1+[A])(1-[A])=1-[A\otimes A]=1
\]
by \eqref{lwax3}.
\end{proof}

\subsection{Construction of the outgoing map}

Similar to Definition \ref{def_outgoing}, we get an outgoing map. For the sake of simplicity (in the literal sense) we restrict our attention to $A$ being a simple abelian variety.

\begin{definition}
Suppose $A$ is a simple abelian variety. We define $r\colon K(\mathsf{Var})_{\mathsf{loc}}\rightarrow K(\mathsf{AV}%
^{\operatorname*{iso}}) \rightarrow K(F)$ to be the composition $p\circ w_{-1}\circ\tilde{R}$, where $p$ is the projector on the summand of $K$-theory corresponding to the isogeny class of $A$ provided by Lemma \ref{lemma_KThyOfAVIsogenyCategory}.
\end{definition}

Finally, we discuss the analogue of Proposition \ref{prop_retract}.

\begin{proposition}\label{prop_retract_av}
The map
\[
r \circ i|_{K(\Oo)^{\circ}_0} \colon K(\Oo)^{\circ}_0 \longrightarrow K(F)^{\circ}_0
\]
induces the same maps on all homotopy groups $\pi_{\ast }$ as
\[
\varphi \colon K(\Oo)\longrightarrow K(F)\qquad M\longmapsto M\otimes_{\Oo}
F,
\]
restricted to the connected component of $0$.
\end{proposition}
\begin{proof}
The tensoring functor is exact and induces on $K$-theory the same map as interpreting a matrix with entries in $\Oo$ as one with entries in $F$.
\[
\xymatrix{
\BGL(\Oo) \ar[r]^{\Oo \hookrightarrow F } \ar[d]_k & \BGL(F) \ar[d]^k \\
K(\Oo)^{\circ}_0 \ar[r]_{\varphi } & K(F)^{\circ}_0 \text{.}
}
\]
Consider
\begin{equation}\label{eqn2}
\xymatrix{
 & \BGL_n(\Oo) \ar[ld] \ar[d]^{\tilde{i}} \ar[r]^-{R} & {\mathsf{\tilde{T}}}^{\times} \ar[r]^-{w_{-1}} \ar[d] & F_f(F)^{\times} \ar[d] \\
K(\Oo)_0^{\circ} \ar[r]^-i & K(\Var)_{\loc, \mathbf{1}} \ar[r]^-{\tilde{R}} & K( {\mathsf{\tilde{T}}} )_{\loc, \mathbf{1}} \ar[r]^-{w_{-1}} & K(F)^{\circ}_0
}
\end{equation}
in $\Ho(\Spaces)$.
The triangle on the left stems from Lemma \ref{lemma:comm3_av}, restricted to a specific $\BGL_n$.
The middle and right square come from Lemma \ref{lemma:comm1}. All downward arrows shift to the connected components of the direct sum neutral element $0$ (resp. the tensor units $\mathbf{1}$).

 The proof of the following claim is shown analogously to Claim \ref{claim} using the subtractive category $\bip_{\Var}(A)$ instead of $\bip_{\Var}(\Gb_m)$.

\begin{claim}\label{claim_av}
The composition of the bottom row is a morphism of $H$-spaces with respect to direct sums.
\end{claim}

 When we handled the $\Gb_m$-case, we used the auxiliary rig category $\Tf$, which was set up as the smallest symmetric bi-monoidal full subcategory of $\MHS_{\Zb}^{\eff}$, which contains $0, \mathbf{1}, \Lb$. Essentially, we cut out precisely the weight $0$ and $-2$ parts relevant for our computation. Then we defined $\Tf'$ to be
$\Tf / (\Tf\cap \mathbf{S}_{-3})$ in order to get rid of higher powers of $\Lb$ of weight $-4$ and lower. As the counterpart for this category we now take the smallest symmetric bi-monoidal full subcategory of ${\mathsf{\tilde{T}}}$ which contains
\[
0\text{,}\qquad\mathbf{1}_{\mathsf{\tilde{T}}}=[\mathbb{Z}\overset
{0}{\longrightarrow}0]\text{,}\qquad\lbrack0\overset{0}{\longrightarrow}A]
\]
so that $R(A)=\mathbf{1}_{\mathsf{\tilde{T}}}+[0\overset{0}{\longrightarrow
}A]$.
There is no need for an analogue quotienting out something like ${\Tf\cap \mathbf{S}_{-3}}$ because the higher powers of the weight $-1$ object $\lbrack0\overset{0}{\longrightarrow}A]$ are already zero by \eqref{lwax3}.

We have $R(A)=[\mathbb{Z}\overset{0}{\longrightarrow}A]$, therefore%
\begin{align*}
pw_{-1}\left(  R(A^{n})\right)    & =pw_{-1}\left(  [\mathbb{Z}\overset
{0}{\longrightarrow}A]^{\otimes n}\right)  \underset{\text{Eq. \ref{lwax3}}%
}{=}pw_{-1}\left(  [\mathbb{Z}\overset{0}{\longrightarrow}A\oplus\cdots\oplus
A]\right)  \\
& =p(A\oplus\cdots\oplus A)=F^{n}\text{.}%
\end{align*}
An automorphism of $A^{n}$, given by an element of $\operatorname*{GL}%
_{n}(\mathcal{O})$, is mapped to the same matrix in $\operatorname*{GL}%
_{n}(F)$. The rest of the argument follows the proof of Proposition \ref{prop_retract}.
\end{proof}

\section{Proof of the second theorem}

\begin{theorem}
\label{thm2}Let $F$ be a CM field with ring of integers $\mathcal{O}$. Let $A$ be a simple
abelian variety over a perfect field $k$ (possibly of positive characteristic) with
complex multiplication by $\mathcal{O}$. Let $\mathsf{Var}$ denote varieties
over $k$.

\begin{enumerate}
\item For every $n\geq1$ the group $\pi_{n}\left(  K(\mathsf{Var})
[A^{-1}]\right)$ has the Quillen $K$-group $K_{n}(\mathcal{O})$ as a direct summand.

\item The rational homotopy type $L_{\mathbb{Q}}K(\mathcal{O})$ is a retract
of $L_{\mathbb{Q}}\left(  K(\mathsf{Var})[A^{-1}]\right)  $ after going to the
$1$-connective cover.
\end{enumerate}
\end{theorem}

\begin{example}\label{avwithcm}
Given an arbitrary CM field $F$ of degree $2m:=[F:\mathbb{Q}]$ (and therefore
$m$ complex places), the Minkowski embedding%
\[
F\hookrightarrow\mathbb{C}^{m}%
\]
gives rise to the quotient $\mathbb{C}^{m}/\mathcal{O}$. This is an abelian
variety\footnote{At first just a complex torus, but one can show the existence of an ample line bundle.} over $\mathbb{C}$ with complex multiplication by $\mathcal{O}$, so the
existence of such $A$ over the complex numbers is immediate (see \cite[Ch. II,
\S 5.2, Theorem 1 and \S 6.1]{MR1492449} and the discussion before the theorem). After this use of the complex numbers, one can show that any such example actually already works over a sufficiently large finite extension of the rationals (\cite[Ch. III, \S 12.4, Proposition 26]{MR1492449}.
\end{example}

\begin{proof}
Below, we write $(\ast)$ for the arrow which is induced from the exact
functor $M\mapsto M\otimes_{\mathcal{O}}F$ on $K$-groups. By a theorem of
Soul\'{e} \cite[Chapter 6, Theorem 6.8]{K-book} the localisation sequence for
$K$-theory gives rise to exact sequences%
\begin{align*}
0  & \longrightarrow K_{n}(\mathcal{O})\overset{(\ast)}{\longrightarrow}%
K_{n}(F)\longrightarrow\bigoplus_{\mathfrak{p}}K_{n-1}(\kappa(\mathfrak{p}%
))\longrightarrow0\qquad\text{(for }n\geq2\text{ even)}\\
0  & \longrightarrow K_{n}(\mathcal{O})\underset{\sim}{\overset{(\ast
)}{\longrightarrow}}K_{n}(F)\longrightarrow0\qquad\text{(for }n\geq3\text{
odd)}%
\end{align*}
where $\mathfrak{p}$ runs through the non-zero prime ideals of $\mathcal{O}$
and $\kappa(\mathfrak{p}):=\mathcal{O}/\mathfrak{p}$ is the respective residue
field. In particular, $(\ast)$ is injective for all $n\geq1$ (the case $n=1$
is easy). By Proposition \ref{prop_retract_av} the map $r\circ i\mid_{K(\mathcal{O}%
)_{0}^{\circ}}$ induces for $n\geq1$%
\[
K_{n}(\mathcal{O})\overset{i}{\longrightarrow}{\pi }_{n}K(\mathsf{Var}%
)[A^{-1}]\overset{r}{\longrightarrow}K_{n}(F)
\]
the same map as $(\ast)$. Thus, its image is again $K_{n}(\mathcal{O})$ and we
have obtained a retraction. This proves (1). Next, note that $\kappa
(\mathfrak{p})$ is always a finite field. Quillen's computation of the
$K$-theory of finite fields shows that $K_{n}(\mathbb{F}_{q})\otimes
\mathbb{Q}=0$ for all $n\geq1$. Thus, we have%
\[
K_{n}(\mathcal{O})\otimes\mathbb{Q}\underset{\sim}{\overset{(\ast
)}{\longrightarrow}}K_{n}(F)\otimes\mathbb{Q}%
\]
for all $n\geq2$. Hence, for the $1$-connective cover we obtain by Whitehead's
lemma that the arrow $(\ast)$ induces an equivalence $L_{\mathbb{Q}%
}K(\mathcal{O})\rightarrow L_{\mathbb{Q}}K(F)$ of homotopy types. Thus, the
map $r\circ i\mid_{K(\mathcal{O})_{0}^{\circ}}$ of Proposition \ref{prop_retract_av}
exhibits the rationalisation as a genuine retract of $L_{\mathbb{Q}}\left(  K(\mathsf{Var}%
)[A^{-1}]\right)  $  in the rational homotopy category.
\end{proof}

\begin{example}\label{example_bigcmfields}
The cyclotomic fields $\Qb(\zeta_n)$ are CM fields with maximal totally real subfield ${\Qb(\zeta_n+\zeta_n^{-1})}$.
\end{example}

\begin{corollary}
Suppose $k=\mathbb{C}$. We have $\dim\left(  K_{n}(\mathsf{Var})\otimes
\mathbb{Q}\right)  =\infty$ for all odd $n\geq3$.
\end{corollary}

\begin{proof}
Pick any $N\geq1$. Let $F$ be a CM\ field of degree $2N$ (such always exist; e.g. an arbitrarily large supply stems from Example \ref{example_bigcmfields})
and let $A$ be a simple abelian variety with complex multiplication by the
ring of integers $\mathcal{O}$ of $F$ (which exists, see Example \ref{avwithcm}). We note that $F$ has no real places and
$N$ complex places, so $r_{2}=N$. By Borel's rank computations we have
$\dim\left(  K_{n}(\mathcal{O})\otimes\mathbb{Q}\right)  =N$ for all
$n\equiv1,3\operatorname{mod}4$ and $n\geq2$ (\cite[Chapter IV, Theorem 1.18
for $A=F$]{K-book}). Thus, by Theorem \ref{thm2} we obtain a copy of
$\mathbb{Q}^{N}$, say with basis $b_{1},\ldots,b_{N}$, as a direct summand of
$\pi_{n}\left(  K(\mathsf{Var})[A^{-1}]\right)\otimes\mathbb{Q}$. Thus, for
$M$ big enough, $A^{\otimes M}\cdot\left\langle b_{1},\ldots,b_{N}%
\right\rangle $ lie in $\pi_{n}\left(  K(\mathsf{Var})\right)  \otimes
\mathbb{Q}$. Now let $N$ go to infinity.
\end{proof}

Many variations are possible and we will not try to be exhaustive.

\begin{theorem}\label{thm3}
Let $A$ be a simple
abelian variety over a perfect field $k$ with complex multiplication by some order $\mathcal{O} \subseteq F$, where $F$ is a division algebra (e.g. an order in a quaternion algebra). Then the rational homotopy type $L_{\mathbb{Q}}K(\mathcal{O})$ is a retract
of $L_{\mathbb{Q}}\left(  K(\mathsf{Var})[A^{-1}]\right)  $ after going to the
$1$-connective cover.
\end{theorem}

\begin{proof}
The proof is as for Theorem \ref{thm2}. Our constructions were sufficiently general to handle non-commutative orders as well. Then use \cite[Chapter IV, Theorem 1.17 and 1.18]{K-book} to see that the $K$-theory of the order rationally and in degrees $\geq 2$ agrees with the $K$-theory of the division algebra $F$, and further with the $K$-theory of the centre of the division algebra (which is a number field). Hence, the same proof as for Theorem \ref{thm2} works.
\end{proof}

For constructing abelian varieties with sufficiently many complex
multiplications of prescribed type for base fields $k$ in positive characteristic, consult
\cite{MR949273}.


\newif\ifabfull\abfulltrue\def\cprime{$'$}
\providecommand{\bysame}{\leavevmode\hbox to3em{\hrulefill}\thinspace}
\providecommand{\MR}{\relax\ifhmode\unskip\space\fi MR }
\providecommand{\MRhref}[2]{%
  \href{http://www.ams.org/mathscinet-getitem?mr=#1}{#2}
}
\providecommand{\href}[2]{#2}


\end{document}